\documentclass[a4paper,11pt]{amsart}
\usepackage[left=2.7cm,right=2.7cm,top=3.5cm,bottom=3cm]{geometry}

\usepackage{amssymb,latexsym,amsmath,amsthm,amscd}
\usepackage{graphicx}
\usepackage{dsfont}
\usepackage[all]{xy}
\usepackage{mathrsfs}
\usepackage{color}
\definecolor{Blue}{rgb}{0.3,0.3,0.9}

\usepackage[T2A,OT1]{fontenc}
\usepackage{amscd}
\DeclareSymbolFont{cyrillic}{T2A}{cmr}{m}{n}
\DeclareMathSymbol{\sha}{\mathalpha}{cyrillic}{216}

\newcommand{\sk}{\vspace{0.1in}}
\newtheorem{thm}{Theorem}[section]
\newtheorem{def-thm}[thm]{Definition-Theorem}
\newtheorem{cor}[thm]{Corollary}
\newtheorem{lem}[thm]{Lemma}
\newtheorem{def-lem}[thm]{Definition-Lemma}
\newtheorem{prop}[thm]{Proposition}

\newtheorem*{mainthm1}{Theorem A}
\newtheorem*{mainthm2}{Theorem B}
\newtheorem*{mainthm3}{Theorem C}

\theoremstyle{definition}
\newtheorem{defn}[thm]{Definition}

\theoremstyle{remark}

\newtheorem{rem}[thm]{Remark}

\numberwithin{equation}{section}
\newcommand{\cD}{\mathcal{D}}

\newcommand{\cT}{\mathcal{T}}

\newcommand{\cO}{\mathcal{O}}

\newcommand{\cS}{\mathcal{S}}

\newcommand{\cJ}{\mathcal{J}}

\newcommand{\cR}{\mathbb{I}}

\newcommand{\nn}{\mathfrak{n}}
\newcommand{\pp}{\mathfrak{p}}
\newcommand{\qq}{\mathfrak{q}}

\newcommand{\fa}{\mathfrak{a}}

\newcommand{\cI}{\mathcal{I}}
\newcommand{\bQ}{\mathbf{Q}}
\newcommand{\bZ}{\mathbf{Z}}
\newcommand{\bC}{\mathbf{C}}

\newcommand{\bT}{\mathbb{T}}

\newcommand{\bchi}{{\chi}}

\newcommand{\ac}{K}

\newcommand{\F}{{\mathbf{f}}}

\newcommand{\fil}{{\mathscr{F}}}

\def\x{\xi}
\def\bx{\boldsymbol{\xi}}
\def\bchi{\boldsymbol{\chi}}
\def\ro{\mathfrak{O}}
\def\ac{\Gamma}

\def\k{\nu}

\newcommand{\unr}{R_0}
\newcommand{\cW}{W}

\begin{document}

\title{On the $p$-adic variation of Heegner points}
%{Heegner cycles and higher weight specializations of big Heegner points, II}
\author[F. Castella]{Francesc Castella}
%\address{
%Department of Mathematics\\ McGill University\\
%Montreal \\ QC H3A-2K6 \\ Canada}
%\curraddr{Department of Mathematics\\ UCLA\\ Los Angeles\\ CA 90095-1555\\ USA}
%\address{Department of Mathematics, UCLA, Los Angeles, CA, 90095-1555, USA}
%\email{castella@math.ucla.edu}
\address{Mathematics Department, Princeton University, Fine Hall, Princeton, NJ 08544-1000, USA}
\email{fcabello@math.princeton.edu}

\thanks{This project has received funding from the European Research Council (ERC) under the European Union's
Horizon 2020 research and innovation programme (grant agreement No. 682152).}

%\thanks{We thank Henri Darmon, Ben Howard, Ming-Lun Hsieh, Antonio Lei,
%Ari Shnidman, Chris Skinner, Eric Urban, and Xin Wan for conversations
%and correspondence related to this work. The results in this paper were first outlined %on April 23, 2012,
%at a workshop on the $p$-adic Langlands program held
%at the Fields Institute, Toronto, in April~2012, and
%we would like to thank the Fields Institute and the organizers of the workshop for their hospitality and support.}

%\subjclass{}
%\keywords{}

%\date{\today}

%\dedicatory{}
%\commby{}

% ----------------------------------------------------------------

\begin{abstract}
In this paper, we prove an ``explicit reciprocity law''
relating Howard's system of big Heegner points 
to a two-variable $p$-adic $L$-function (constructed here)
interpolating the $p$-adic Rankin $L$-series
of Bertolini--Darmon--Prasanna 
in Hida families.
As applications, we obtain a direct relation between classical Heegner cycles and the
higher weight specializations of big Heegner points, refining
earlier work of the author,
and prove the vanishing of Selmer groups 
of CM elliptic curves twisted by $2$-dimensional Artin representations in cases predicted by the equivariant
Birch and Swinnerton-Dyer conjecture.
\end{abstract}

\subjclass[2010]%{11F67, 11F80, 11G40}
{11G05 (Primary); 11G40 (Secondary)}
\maketitle

\setcounter{tocdepth}{1}
\setcounter{section}{0}
\tableofcontents

\section{Introduction}

Let $f=\sum_{n\geqslant 1}a_n(f)q^n\in S_{2r}(\Gamma_0(N))$ be a newform of weight $2r\geqslant 2$, fix
a prime $p\nmid 6N$, and let $L$ be a finite extension of $\bQ_p$ with ring of integers $\ro$
containing the image the Fourier coefficients $a_n(f)$ under a fixed embedding
$\imath_p:\overline{\bQ}\hookrightarrow\overline{\bQ}_p$. Denote by
\[
\rho_f:G_{\bQ}:={\rm Gal}(\overline{\bQ}/\bQ)\longrightarrow{\rm Aut}_L(V_f(r))\simeq{\rm GL}_2(L)
\]
the Kummer self-dual twist of the $p$-adic Galois representation associated with $f$.
Let $K$ be an imaginary quadratic field of odd discriminant $-D_K<-3$ and ring of integers $\mathcal{O}_K$
satisfying the classical \emph{Heegner hypothesis} relative to $N$:
\begin{equation}\label{eq:heeg-hyp}
\textrm{there is an integral ideal $\mathfrak{N}$ of $K$ with $\cO_K/\mathfrak{N}\simeq\bZ/N\bZ$;}\tag{heeg}
\end{equation}
equivalently, every prime $q\mid N$ either splits or ramifies in $K$, with $q^2\nmid N$
in the latter case.
\sk

The first purpose of this paper %(see Theorem~A below)
is to complete earlier work of the author \cite{cas-inv} comparing two natural
constructions of a cohomology class of ``Heegner-type'' attached to the pair $(f,K)$.
For the first of these, let ${\rm Sel}(K,V_f(r))\subseteq H^1_{}(G_K,V_f(r))$
be the Bloch--Kato Selmer group for the representation $V_f(r)\vert_{{\rm Gal}(\overline{\bQ}/K)}$.
By %work of Nekov{\'a}{\v{r}}
\cite{nekovarCRM}, the image under the $p$-adic \'etale Abel--Jacobi map
of classical Heegner cycles \cite{nekovar302} on the $(2r-1)$-dimensional Kuga--Sato variety of level $N$
give rise to a class
\[
\Phi^{\rm{\acute{e}t}}_{f,K}(\Delta^{\rm heeg}_r)\in{\rm Sel}(K,V_f(r)).
\]
%(See $\S\ref{sec:higher}$ for a more detailed review of the construction.)

For the second class, assume that $f$ is \emph{ordinary at $\imath_p$}, i.e.:
\begin{equation}\label{eq:ord-hyp}
a_p(f)\in\ro^\times\tag{ord}.
\end{equation}
Fix a Galois-stable $\ro$-lattice $T_f\subseteq V_f$,
let $\bar{\rho}_f:G_\bQ\rightarrow{\rm GL}_2(\kappa_L)$ be the associated semisimple
residual representation, where $\kappa_L$ is the residue field of $L$, and assume that
\begin{equation}\label{eq:irr-hyp}
\textrm{$\bar\rho_f$ is irreducible}.\tag{irred}
\end{equation}
Let $D_p\subseteq G_\bQ$ be a fixed decomposition group at $p$.
By hypothesis (\ref{eq:ord-hyp}), the restriction
$\rho_f\vert_{D_p}$ can be made upper-triangular, and we shall assume in addition that
\begin{equation}\label{eq:dist-hyp}
\textrm{$\bar\rho_f$ is $D_p$-distinguished};\tag{dist}
\end{equation}
i.e., the semisimplification of $\bar{\rho}_f\vert_{D_p}$
is the direct sum of two \emph{distinct} characters.
%For simplicity,
Suppose that $r\equiv 1\pmod{p-1}$, %\footnote{This is largely for simplicity.},
and let $\mathbf{f}=\sum_{n\geqslant 1}\mathbf{a}_nq^n\in\cR[[q]]$
be the Hida family passing though $f$. Thus $\cR$ is a finite flat
extension of $\ro[[X]]$, and for every continuous $\ro$-algebra homomorphism $\k:\cR\rightarrow\overline{\bQ}_p$
satisfying $\k(1+X)=(1+p)^{k_\k-2}$ for some integer $k_\k\geqslant 2$ with $k_\k\equiv 2\pmod{p-1}$,
the $q$-series $\mathbf{f}_\k:=\sum_{n\geqslant 1}\k(\mathbf{a}_n)q^n$
is such that
\begin{align*}
\mathbf{f}_\k%&:=\sum_{n\geqslant 1}\k(\mathbf{a}_n)q^n
=f_\k(q)-\frac{p^{k_\k-1}}{\k(\mathbf{a}_p)}f_\k(q^p),
\end{align*}
for some newform $f_\k\in S_{k_\k}(\Gamma_0(N))$,
with $f=f_{\k}$ for a unique $\k$ with $k_\k=2r$.
(More generally, we may consider the image
of $\mathbf{f}$ under maps $\k:\cR\rightarrow\overline{\bQ}_p$ with arbitrary $k_{\k}\in\bZ_{>0}$.)
%, which will be important
%later in this paper, especially for $k_\nu=1$.)
Under the above hypotheses, Howard's
construction of big Heegner points \cite{howard-invmath} produces a class
\[
\mathfrak{Z}_0\in H^1(G_K,\mathbf{T}^\dagger),
\]
where $\mathbf{T}$ is a self-dual twist
%(or the \emph{critical twist}, in the terminology of \emph{loc.cit.})
of the big Galois representation $\mathbf{\rho}_{\mathbf{f}}$
associated with $\mathbf{f}$. Under some additional hypothesis on the residual representation
$\bar{\rho}_{\mathbf{f}}\simeq\bar{\rho}_f$ when $(D_K,N)>1$, one can show that
$\mathfrak{Z}_0$ lies in the so-called \emph{strict Greenberg Selmer group}
${\rm Sel}_{\rm Gr}(K,\mathbf{T}^\dagger)\subseteq H^1(G_K,\mathbf{T}^\dagger)$,
and so its image under the specialization map $\k_f$ produces a second class
$\nu_f(\mathfrak{Z}_0)\in{\rm Sel}(K,V_f(r))$.

\begin{mainthm1}[Theorem~\ref{thm:higher-wt-sp}]
Assume in addition that $p=\pp\overline{\pp}$ splits in $K$, $\bar{\rho}_f\vert_{G_K}$ is irreducible,
and $\bar{\rho}_f$ is ramified at every prime $q\mid N$ nonsplit in $K$. Then
\[
\nu_f(\mathfrak{Z}_0)=\biggl(1-\frac{p^{r-1}}{\nu_f(\mathbf{a}_p)}\biggr)^2
\frac{\Phi^{\rm{\acute{e}t}}_{f,K}(\Delta^{\rm heeg}_r)}{u_K(2\sqrt{-D_K})^{r-1}},
\]
where $u_K=\vert\cO_K^\times\vert/2$.
\end{mainthm1}

This subsumes the main result of \cite{cas-inv}, which only implies the equality in Theorem~A
under the assumption of Howard's ``horizontal nonvanishing conjecture'' \cite[Conj.~2.2.2]{howard-invmath}
and the nondegeneracy of the cyclotomic $p$-adic height pairing. The class $\mathfrak{Z}_0$
is obtained from Howard's big Heegner point $\mathfrak{X}_1$ of conductor $1$, and more generally
Theorem~\ref{thm:higher-wt-sp} establishes a comparison between the Selmer classes constructed from
Heegner cycles of conductor $c\geqslant 1$ prime to $Np$ and the corresponding higher weight specializations of the
big Heegner point $\mathfrak{X}_c$. In particular, Theorem~\ref{thm:higher-wt-sp} answers
a question raised by Howard (see \cite[p.~93]{howard-invmath}).
\sk

Similarly as in \cite{cas-inv}, the proof of Theorem~A follows from relating the cohomology classes involved to %the same
special values of $L$-functions. More precisely, extending work of
Bertolini--Darmon--Prasanna \cite{bdp1} and Brako{\v{c}}evi{\'c} \cite{braIMRN},
in \cite{cas-hsieh1} we constructed an anticyclotomic $p$-adic $L$-function
$\mathscr{L}_{\pp,\psi}(f)$ interpolation central critical values for the Rankin--Selberg
convolution of $f$ with certain theta series. Moreover, also in \cite{cas-hsieh1} we constructed a compatible
system of cohomology classes $\mathbf{z}_f$ interpolating the $p$-adic \'etale Abel--Jacobi
images of (generalized) Heegner cycles of $p$-power conductor, and
extending the $p$-adic Gross--Zagier formula
of \cite{bdp1} %to character of $p$-power conductor,
we obtained an ``explicit reciprocity law''
\begin{equation}\label{eq:ERL-CH}
\langle\mathcal{L}_{\pp,\psi}(\mathbf{z}_f),\omega_f\otimes t^{1-2r}\rangle=-\mathscr{L}_{\pp,\psi}(f)
\end{equation}
relating %the $p$-adic $L$-function
$\mathscr{L}_{\pp,\psi}(f)$ to the image of $\mathbf{z}_f$ under a Perrin-Riou ``big logarithm map''.
In Section~\ref{sec:Lp} of this paper, we construct a two-variable $p$-adic $L$-function $\mathscr{L}_{\pp,\bx}(\mathbf{f})$
interpolating the $p$-adic $L$-functions of \cite{cas-hsieh1} attached to the different
specializations $f_\nu$ of $\mathbf{f}$; in particular,
\begin{equation}\label{eq:esp-L}
\nu_f(\mathscr{L}_{\pp,\bx}(\mathbf{f}))=\mathscr{L}_{\pp,\psi}(f).
\end{equation}
The key new ingredient in our proof of Theorem~A is then the connection that we find
between $\mathscr{L}_{\pp,\bx}(\mathbf{f})$ and the system $\mathfrak{Z}_\infty$
of Howard's big Heegner points of $p$-power conductor.

\begin{mainthm2}[Theorem~\ref{thm:equality}]
There is a Perrin-Riou big logarithm map $\mathcal{L}_{\pp,\bx}$ sending Howard's system of big Heegner points $\mathfrak{Z}_\infty$
attached to $\mathbf{f}$ over the anticyclotomic $\bZ_p$-extension of $K$
to the two-variable $p$-adic $L$-function $\mathscr{L}_{\pp,\bx}(\mathbf{f})$.
\end{mainthm2}
%There is a two-variable Perrin-Riou big logarithm map $\mathcal{L}_{\pp,\psi}$
%such that
%\[
%\langle\mathcal{L}_{\pp,\psi}(\mathfrak{Z}_f),\omega_{\mathbf{f}}\rangle=-\mathscr{L}_{\pp,\psi}(\mathbf{f}).
%]

The construction of the Perrin-Riou map $\mathcal{L}_{\pp,\bx}$ is given in Section~\ref{sec:2varL},
building upon work of Ochiai~\cite{Ochiai-Col} and Loeffler--Zerbes~\cite{LZ2},
and the proof of the ``explicit reciprocity law'' of Theorem~B is obtained in Section~\ref{sec:comparison}
after a suitable extension of the calculations in \cite{cas-inv}. With this result
at hand, the proof of Theorem~A follows easily by specializing the equality in
Theorem~\ref{thm:equality} at $\nu_f$, using $(\ref{eq:esp-L})$ and the interpolation
property of the map $\mathcal{L}_{\pp,\bx}$, and comparing it with the equality in $(\ref{eq:ERL-CH})$.
\sk

%\begin{rem}
%Since the first drafts of this paper were released,
%Theorem~B has played an important role in a number
%of arithmetic applications. %beyond its key role in the proof of Theorem~A.
%In particular, it has been used by X.~Wan in his proof \cite{wan} of
%Perrin-Riou's Heegner point main conjecture  \cite{PR-HP} (see also \cite{cas-BF});
%by D.~Casazza \cite{casazza} in his proof of certain cases of the ``elliptic Stark conjecture'' of
%Darmon--Lauder--Rotger \cite{DLR}; and in the study of
%the specializations of Howard's big Heegner points at exceptional primes in the Hida family
%(see \cite{cas-JIMJ}).
%, leading to some cases a certain $p$-adic variants of the Birch and Swinnerton-Dyer conjecture recently
%formulated (and esblished, in those cases) by D.~Disegni \cite{disegni-BSD}.
%\end{rem}

The second purpose of this paper is to exploit the $p$-adic variation of Heegner points
in Hida families to establish certain new cases of the equivariant Birch--Swinnerton-Dyer conjecture
for rational elliptic curves with complex multiplication. More precisely,
let $A/\bQ$ be an elliptic curve with CM, and let
\[
\varrho:G_\bQ\longrightarrow{\rm Aut}_E(V_\varrho)\simeq{\rm GL}_2(E)
\]
be a $2$-dimensional odd and irreducible Artin representation factoring through
a finite quotient ${\rm Gal}(F/\bQ)$ and with values in a finite extension $E\subseteq\bC$ of $\bQ$.
Let $T_p(A)$ be the $p$-adic Tate module of $A$, and set $V_p(A):=\bQ_p\otimes_{\bZ_p}T_p(A)$.
Associated to the compatible system $V_p(A)\otimes\imath_pV_\varrho$ of $p$-adic representations of $G_\bQ$
is a Artin--Hasse--Weil $L$-function $L(A/\bQ,\varrho,s)$. This is defined for ${\rm Re}(s)>3/2$
by an absolutely convergent Euler product of degree $4$, and by
%\cite{BCDT}\footnote{Since $A$ has CM, a reference to Hecke's work \cite{hecke-CM} would already do here.}
\cite{hecke-CM}
and \cite{KW-serre-I} it is known to admit
an analytic continuation to the entire complex plane, with a functional equation relating
its values at $s$ and $2-s$. The equivariant Birch--Swinnerton-Dyer conjecture then predicts that
\begin{equation}\label{eq:BSD-rank}
{\rm ord}_{s=1}L(A/\bQ,\varrho,s)\overset{?}={\rm dim}_L{\rm Hom}_{G_{\bQ}}(V_\varrho,A(F)_L),
\end{equation}
and that
\begin{equation}\label{eq:BSD-sha}
{\rm Hom}_{G_{\bQ}}(V_\varrho,\sha_{p^\infty}(A/F)_L)\overset{?}=\{0\}
\end{equation}
for all primes $p$, where $\sha_{p^\infty}(A/F)$ is the $p$-primary component of the Tate--Shafarevich
group of $A/F$, and for any abelian group $M$ let set $M_L:=M\otimes_{\bZ}L$.
Let $N_A$ and $N_\varrho$ be the conductor of $A$ and $\varrho$, respectively, and denote
by ${\rm Sel}(F,V_pA)\subseteq H^1(G_F,V_p(A))$ the Bloch--Kato
Selmer group for $V_p(A)\vert_{{\rm Gal}(\overline{\bQ}/F)}$.
%The main result of this paper towards this conjecture can then be stated as follows.

\begin{mainthm3}%[Theorem~\ref{thm:BSD-CM}]
\label{mainthm:BSD}
Let $A/\bQ$ be an elliptic curve of conductor $N_A$ and with complex multiplication by an imaginary quadratic field $K$,
let $p\nmid 6N_\varrho N_A$ be a prime, and let $\mathfrak{P}$ be a prime of $E$ above $p$.
Assume that:
\begin{itemize}
\item{} $N_\varrho$ and $N_A$ are coprime;
\item{} $p=\pp\overline\pp$ splits in $K$;
\item{} $K$ satisfies hypothesis {\rm (\ref{eq:heeg-hyp})} relative to $N_\varrho$;
\item{} $\varrho({\rm Frob}_p)$ has distinct eigenvalues modulo $\mathfrak{P}$.
\end{itemize}
If $L(A/\bQ,\varrho,1)\neq 0$, then
\[
{\rm Hom}_{G_\bQ}(V_\varrho,{\rm Sel}(F,V_p(A))_L)=\{0\}.
\]
In particular, $(\ref{eq:BSD-rank})$ and $(\ref{eq:BSD-sha})$ hold.
\end{mainthm3}

The conclusion that $(\ref{eq:BSD-rank})$ holds under the nonvanishing of $L(A/\bQ,\varrho,s)$ at $s=1$
was already contained in earlier work of Bertolini--Darmon--Rotger \cite[Thm.~A]{BDR2},
while recent work of Kings--Loeffler--Zerbes \cite[Thm.~11.7.4]{KLZ2} establishes an analog of Theorem~C
for rational elliptic curves \emph{without} complex multiplication (the CM case
is excluded in \cite{KLZ2} by the ``big image hypothesis'' of [\emph{loc.cit.}, \S{11.1}]).
Thus the new content of Theorem~C is the vanishing of the $\varrho$-isotypical
component of $\sha_{p^\infty}(A/F)_L$ for ``half'' of the primes $p$ under the nonvanishing of $L(A/\bQ,\varrho,1)$.
\sk

Let us conclude this Introduction with a few words about the proof of Theorem~C. Denote
by $L(f/K,\chi,s)$ the Rankin--Selberg $L$-function for the convolution of a cusp form
$f\in S_k(\Gamma_1(N))$ with a Hecke character $\chi$ of $K$.
From the explicit reciprocity law of Theorem~B, we deduce a proof of the implication
\[
L(f_\nu/K,\chi\mathbf{N}^{k_\nu/2},0)\neq 0\;\Longrightarrow\;\nu(\mathfrak{Z}_\infty)^{\chi^{-1}}\neq 0,
\]
for $\nu:\cR\rightarrow\overline{\bQ}_p$ of weight $k_\nu>0$
and certain anticyclotomic Hecke characters $\chi$. Since Howard's
big Heegner points satisfies the axioms of an %anticyclotomic
Euler system,
%\footnote{And in fact, Theorem~A combined with the result of \cite{cas-hsieh1} give another proof of this statement.},
one can deduce from Kolyvagin's methods %of Euler systems
(as extended in \cite[\S{7.2}]{cas-hsieh1} to the anticyclotomic setting) a proof of the implication
\[
L(f_\nu/K,\chi\mathbf{N}^{k_\nu/2},0)\neq 0\;\Longrightarrow\;{\rm Sel}(K,V_{\nu,\chi})=\{0\},
\]
where ${\rm Sel}(K,V_{\nu,\chi})$ is the Bloch--Kato Selmer group for the representation
$V_{f_\nu}\vert_{G_K}\otimes\chi$. Since by %the work of Khare--Winterberger
\cite{KW-serre-I} any Artin representation $\varrho$ as in Theorem~C is attached to some $g\in S_1(\Gamma_1(N_\varrho))$,
taking $\chi$ so that $\chi\mathbf{N}^{1/2}$ corresponds to the grossencharacter of $A$, %by the theory of complex multiplication,
$\mathbf{f}$ to be a Hida family passing through $g$, %one of the $p$-stabilization of $f$,
and specializing the resulting $\mathfrak{Z}_\infty$ to weight one, the proof of Theorem~C follows.
\sk

\noindent\emph{Some notations and definitions.}
%For every prime $p$, we fix once and for all embeddings $\imath_\infty:\overline{\bQ}\hookrightarrow\bC$
%and $\imath_p:\overline{\bQ}\hookrightarrow\overline{\bQ}_p$.
For any place $v$ of number field $E$,
we let ${\rm rec}_v:E_v^\times\rightarrow G_{E_v}^{\rm ab}$ and
${\rm rec}_E:E^\times\backslash\mathbb{A}_E^\times\rightarrow G_E^{\rm ab}$
be the local and global reciprocity map, respectively, with geometric
normalizations. If $\phi:\bZ_p^\times\rightarrow\bC^\times$
is a continuous character of conductor $p^n$, the Gauss sum of $\phi$ is defined by
\[
\mathfrak{g}(\phi)=\sum_{u\in(\bZ/p^n\bZ)^\times}\phi(u)\boldsymbol{\rm e}(u/p^n),
\]
where $\boldsymbol{\rm e}(z)={\rm exp}(2\pi iz)$, and if $\chi:\bQ_p\rightarrow\bC^\times$
is a continuous character of conductor $p^n$, we define the $\varepsilon$-factor of $\chi$ by
$\varepsilon(\chi)=p^n\chi^{-1}(p^n)\mathfrak{g}(\chi^{-1})^{-1}$.

% ----------------------------------------------------------------

\section{$p$-adic Rankin $L$-series}\label{sec:Lp}

In this section, %after a brief review of the necessary background,
we give the construction of a two-variable
anticyclotomic $p$-adic $L$-function $\mathscr{L}_{\pp,\bx}(\mathbf{f})$ attached to a Hida family $\mathbf{f}$ and
an imaginary quadratic field $K$ in which $p=\pp\overline{\pp}$ splits.
Such construction %of $\mathscr{L}_{\pp,\bx}(\mathbf{f})$
closely parallels the one-variable construction by Brako{\v{c}}evi{\'c}~\cite{braIMRN},
and was essentially contained in \cite{bra2}.
%We largely follow the exposition in \cite[\S\S{2-3}]{cas-hsieh1}, referring
%the reader to \emph{loc.cit.} and the references therein for further details.

\subsection{Geometric and $p$-adic modular forms}\label{sec:modular}

Fix a prime $p$, and let $N\geqslant 3$ be an integer prime to $p$.

\begin{defn}\label{def:GMF}
Let $k$ be an integer and let $B$ be a $\bZ_{(p)}$-algebra. A \emph{geometric modular form} $f$
of weight $k$ on $\Gamma_1(Np^\infty)$ defined over $B$ is a rule which assigns to every triple
$(A,\eta,\omega)_{/C}$, over an arbitrary $B$-algebra $C$, consisting of:
\begin{itemize}
\item{} an elliptic curve $A/C$;
\item{} a $\Gamma_1(Np^\infty)$-level structure $\eta$ on $A$, i.e., an immersion
\[
\eta=(\eta^{(p)},\eta_p):\boldsymbol{\mu}_N\oplus\boldsymbol{\mu}_{p^\infty}\hookrightarrow A[N]\oplus A[p^\infty]
\]
as group schemes over $S$;
\item{} a $C$-basis $\omega$ of $H^0(A,\Omega^1_{A/C})$,
\end{itemize}
a value $f(A,\eta,\omega)\in C$ depending only on the isomorphism class of $(A,\eta,\omega)$ over $C$
and such that:
\begin{enumerate}
\item{} For any $B$-algebra homomorphism $\varphi:C\rightarrow C'$, we have
\[
f((A,\eta,\omega)\otimes_{C}C')=\varphi(f(A,\eta,\omega));
\]
\item{} For all $\lambda\in C^\times$, we have
\[
f(A,\eta,\lambda\omega)=\lambda^{-k}f(A,\eta,\omega);
\]
\item{} Letting $({\rm Tate}(q),\eta_{\rm can},\omega_{\rm can})_{/B((q))}$ be
the Tate elliptic curve $\mathbf{G}_m/q^{\bZ}$
equipped with its canonical level structure $\eta_{\rm can}$ and differential $\omega_{\rm can}$, we have
\[
f({\rm Tate}(q),\eta_{\rm can},\omega_{\rm can})\in B[[q]].
\]
\end{enumerate}
\end{defn}

Let ${\rm Ig}(N)_{/\bZ_{(p)}}$ be the Igusa scheme
parameterizing pairs $(A,\eta)_{/S}$ of elliptic curves equipped with $\Gamma_1(Np^\infty)$-level structure over
arbitrary locally Noetherian $\bZ_{(p)}$-schemes $S$. %The group $\bZ_p^\times$ acts on ${\rm Ig}_{/\bZ_{(p)}}$ by the rule
%\[
%\langle u\rangle(A,\eta)=(A,\eta^{(p)},\eta_pu)
%\]
%for all $[(A,\eta)]=[(A,\eta^{(p)},\eta_p)]\in{\rm Ig}(N)_{/\bZ_{(p)}}$.
The generic fiber ${\rm Ig}(N)_{/\bQ}$ is given by
\begin{equation}\label{eq:perfectoid}
{\rm Ig}(N)_{/\bQ}=\varprojlim_sY_1(Np^s)_{/\bQ},
\end{equation}
where $Y_1(Np^s)_{/\bQ}$ is the usual open modular curve of level $\Gamma_1(Np^s)$,
and a geometric modular form $f$ as in Definition~\ref{def:GMF}
can be viewed as a section of a certain sheaf on ${\rm Ig}(N)_{/\bZ_{(p)}}$.

\begin{defn}
Denote by $\Gamma^{\rm wt}$ the group $1+p\bZ_p\subseteq\bZ_p^\times$,
let $k\in\bZ_p$ and let $\varepsilon:\Gamma^{\rm wt}\rightarrow\boldsymbol{\mu}_{p^\infty}$
be a finite order character. A \emph{$p$-adic modular form} of tame level $N$ and weight $(k,\varepsilon)$,
defined over a $p$-adic ring $R$ (i.e., $R\simeq\varprojlim_m R/p^mR$) is a function $f$ on the formal completion
$\widehat{\rm Ig}(N)_{/R}$ of ${\rm Ig}(N)_{/R}$ satisfying
\[
f\vert\langle u\rangle(A,\eta):=f(A,\eta^{(p)},\eta_pu)=\varepsilon(u)u^kf(A,\eta),
\]
for all $u\in\Gamma^{\rm wt}$ and $[(A,\eta)]=[(A,\eta^{(p)},\eta_p)]\in\widehat{\rm Ig}(N)_{/R}$.
\end{defn}

Denote by $V_p(N;R)$ the space of $p$-adic modular forms of tame level $N$ %(and arbitrary weight)
defined over a $p$-adic ring $R$. Associated with every geometric modular form $f$
on $\Gamma_1(Np^\infty)$ over $R$ there is a $p$-adic modular form $\widehat{f}\in V_p(N;R)$ %, with the same weight as $f$,
defined by the rule
\[
\widehat{f}(A,\eta)=f(A,\eta,\widehat{\omega}(\eta_p)),\quad\textrm{for all $[(A,\eta)]\in\widehat{\rm Ig}(N)_{/R}$},
\]
where $\widehat{\omega}(\eta_p)$ is the canonical differential on $A$ arising from the isomorphism
of formal groups $\widehat{\eta}_p:\widehat{\mathbf{G}}_m\simeq\widehat{A}$ induced by
$\eta_p:\boldsymbol{\mu}_{p^\infty}\hookrightarrow A[p^\infty]$.

\subsection{$\cR$-adic modular forms}
\label{subsec:crit}

Let $\ro$ be the ring of integers of a finite extension of $L/\bQ_p$, and set
$\Lambda^{\rm wt}_{\ro}=\ro[[\Gamma^{\rm wt}]]$. For any $k\in\bZ$ and
$\varepsilon:\Gamma^{\rm wt}\rightarrow\boldsymbol{\mu}_{p^\infty}$ let
\[
\k_{k,\varepsilon}:\Lambda_{\ro}^{\rm wt}\longrightarrow\overline{\bQ}_p
\]
be the $\ro$-algebra homomorphism defined by $u\mapsto\varepsilon(u)u^{k-2}$
for $u\in\Gamma^{\rm wt}$.

\begin{defn}\label{def:arith-primes}
Let $\cR$ be a $\Lambda^{\rm wt}_{\ro}$-algebra. We say that
a $\ro$-algebra homomorphism $\k:\cR\rightarrow\overline{\bQ}_p$ is an \emph{arithmetic prime}
if the composition
\[
\Lambda_{\ro}^{\rm wt}\longrightarrow\cR\overset{\k}\longrightarrow\overline{\bQ}_p
\]
is of the form $\k_{k,\varepsilon}$ for some $k\in\bZ_{\geqslant 2}$ and
$\varepsilon:\Gamma^{\rm wt}\rightarrow\boldsymbol{\mu}_{p^\infty}$.
%The pair $(k,\varepsilon)$ is then clearly unique, and
We then say that $\k$ has weight
$(k,\varepsilon)$, omitting $\varepsilon$ from the notation if $\varepsilon=\mathds{1}$.
%wild level $s$, and wild character $\psi_\k$, where
%\[
%\psi_\k:(\bZ/p^s\bZ)^\times\twoheadrightarrow(\bZ/p^{s-1}\bZ)\simeq 1+p\bZ_p/1+p^s\bZ_p
%\longrightarrow\overline{\bQ}_p^\times
%\]
%is induced by the map $1+p\bZ_p\rightarrow\overline{\bQ}_p^\times$ sending $1+p\mapsto\zeta$.
\end{defn}

We denote by $\mathcal{X}_{\ro}^a(\cR)$ the set of arithmetic primes of $\cR$, which
we may also view as a subset of ${\rm Spf}(\cR)(\overline{\bQ}_p)$, and
for each $\k\in\mathcal{X}_{\ro}^a(\cR)$ let $F_\k$ be the residue field of ${\rm ker}(\k)$,
and $\cO_\k\subseteq F_\k$ be the valuation ring.

\begin{defn}\label{def:hida}
Let $\psi_0:(\bZ/Np\bZ)^\times\rightarrow\ro^\times$ be a Dirichlet character modulo $Np$,
and let $\cR$ be a finite flat $\Lambda_{\ro}^{\rm wt}$-algebra.
\begin{enumerate}
\item{}
An \emph{$\cR$-adic modular form} of tame level $N$ and character $\psi_0$
is a formal $q$-expansion
\[
\mathbf{f}=\sum_{n=0}^\infty\mathbf{a}_nq^n\in\cR[[q]]
\]
such that for all %but finitely many
$\k\in\mathcal{X}_{\ro}^a(\cR)$ of weight $(k,\varepsilon)$,
the $q$-series $\sum_{n}\k(\mathbf{a}_n)q^n$ %\in F_\k[[q]]$
is the $q$-expansion of a $p$-adic modular form $\mathbf{f}_\k\in V_p(N;\cO_\k)$
of weight $(k,\varepsilon)$.
Denote by $G(N,\psi_0;\cR)$ the $\cR$-module of such formal $q$-expansions.

\item{} We say that $\mathbf{f}\in G(N,\psi_0;\cR)$
is \emph{arithmetic} if, for all $\k\in\mathcal{X}_{\ro}^a(\cR)$,
$\mathbf{f}_\k=\widehat{f}_\k$ is the $p$-adic avatar of a classical modular form
\[
f_\k\in M_k(\Gamma_0(Np^s),\psi_0\varepsilon\omega^{2-k}),
\]
where $s>0$ is the power of $p$ in the conductor of $\varepsilon$,
%(viewed as a Dirichlet character in the obvious manner),
and $\omega:(\bZ/p\bZ)^\times\rightarrow\bZ_p^\times$ is the Teichm\"uller character;
and say that $\mathbf{f}$ is \emph{cuspidal} if $f_\k$ is a cusp form for all $\k\in\mathcal{X}_{\ro}^a(\cR)$.
Denote by $S^a(N,\psi_0;\cR)\subseteq G(N,\psi_0;\cR)$ the submodule consisting of
cuspidal arithmetic $\cR$-adic modular forms.

\item{} We say that $\mathbf{f}\in S^a(N,\psi_0;\cR)$ is \emph{ordinary} if the
$U_p$-operator acts invertibly on $\mathbf{f}_\k$
for all $\k\in\mathcal{X}_{\ro}^a(\cR)$, and let $S^{\rm ord}(N,\psi_0;\cR)\subseteq S^a(N,\psi_0;\cR)$
be the corresponding submodule.
\end{enumerate}
\end{defn}

%We now relate the $\cR$-module $G(N,\psi_0;\cR)$
Define
\[
V_p(N,\psi_0;\cR):=V_p(N,\psi_0;\ro)\widehat{\otimes}_\ro\cR,
\]
and let $[z]:\bZ_p^\times\rightarrow\ro[[\bZ_p^\times]]^\times$ be the character
given by inclusion of group-like elements. The space $V_p(N,\psi_0;\cR)$ is thus equipped
with two different actions of $z\in\Gamma^{\rm wt}$:
one via the diamond operators $\langle z\rangle$ acting on the first factor of the above completed
tensor product, and the other via multiplication by $[z]$ on the second factor.

\begin{prop}
There is a canonical $\cR$-module isomorphism
\[
G(N,\psi_0;\cR)=\{\mathbf{f}\in V_p(N,\psi_0;\cR)\;\colon\;\mathbf{f}\vert\langle z\rangle=[z]\mathbf{f},
\quad\forall z\in\bZ_p^\times\}.
\]
\end{prop}

\begin{proof}
See %\cite[Prop.~7.1]{bra2}.
\cite[Thm.~3.2.16]{Hida-GMF}.
\end{proof}

Thus by extension of scalars we may evaluate any $\mathbf{f}\in G(N,\psi_0;\cR)$
at a point $x\in\widehat{\rm Ig}(N)(\cR)$, producing an element
$\mathbf{f}(x)\in\cR$. This will be used in $\S\ref{sec:measures}$
%the next section
to define measures associated with $\mathbf{f}$ which, for appropriate choices of $x$
(defined in $\S\ref{sec:CMpoints}$), % attached to an imaginary quadratic field $K$),
interpolate special values of $L$-functions.

\subsection{Modular measures}\label{sec:measures}

Let $\cW$ be a finite extension of $\widehat{\bZ}_p^{\rm ur}$,
the completion of the ring of integers of the maximal unramified extension of $\bQ_p$,
and denote by ${\rm Cont}(\bZ_p,\cW)$ the space of continuous $\cW$-valued functions on $\bZ_p$.
 %equipped with the norm $\norm{\varphi}=\sup_{z\in\bZ_p}\vert\varphi(z)\vert_p$.
Let
\[
{\rm Meas}(\bZ_p,\cW):={\rm Hom}_{\rm cts}({\rm Cont}(\bZ_p,\cW),\cW)
\]
be the space of $\cW$-valued measures of $\bZ_p$.
 %defined as the continuous $\cW$-linear dual of ${\rm Cont}(\bZ_p,\cW)$.
As usual, if $\mu\in{\rm Meas}(\bZ_p,\cW)$ and $\phi\in{\rm Cont}(\bZ_p,\cW)$,
we use the notation $\int_{\bZ_p}\phi(z){\rm d}\mu(z):=\mu(\phi)$.
\sk

The \emph{Amice transform} $\mathscr{A}_\mu(T)\in\cW[[T]]$ of a measure $\mu$
is the power series
\[
\mathscr{A}_\mu(T):=\sum_{m=0}^\infty c_m(\mu)T^m,
%\quad\textrm{where $c_m(\mu)=\int_{\bZ_p}\binom{z}{m}\mu(z)$}.
\]
where $c_m(\mu)=\int_{\bZ_p}\binom{z}{m}{\rm d}\mu(z)$.
One easily checks that
\[
\int_{\bZ_p}z^n{\rm d}\mu(z)=\left(T\frac{d}{dT}\right)^n\mathscr{A}_{\mu}(T)\bigr\vert_{T=0}
\]
for all $n\geqslant 0$, and by Mahler's theorem the rule $\mu\mapsto\mathscr{A}_\mu$
defines an isomorphism ${\rm Meas}(\bZ_p,\cW)\simeq\cW[[T]]$ of $p$-adic Banach algebras.

%Equipped with the norm $\norm{\mu}=\sup_{\norm{\varphi}=1}\vert\int_{\bZ_p}\varphi(z)\mu(z)\vert_p$,
%the space ${\rm Meas}(\bZ_p,\cW)$ becomes a $p$-adic Banach algebra with the convolution product.
%
%In particular, any measure on $\bZ_p$ is uniquely determined by its values
%on the functions $\binom{z}{m}$ for all $m\geq 0$.
\sk

For $\cR=\Lambda_{\ro}^{\rm wt}$, any $\cR$-adic modular form
$\mathbf{f}=\sum_{n=0}^\infty\mathbf{a}_nq^n\in G(N,\psi_0;\cR)$
defines a measure $\mu_{\mathbf{f}}$ on $\Gamma^{\rm wt}$
with values in $V_p(N,\psi_0;\ro)$ by
\[
\int_{\Gamma^{\rm wt}}z^k{\rm d}\mu_{\mathbf{f}}(z)=\sum_{n=0}^\infty
\k_k(\mathbf{a}_n)q^n,
\]
for all $k\in\bZ$; by linearity, the same is true for modular forms over
any finite $\Lambda_{\ro}^{\rm wt}$-algebra $\cR$.
\sk

Let $d$ be the operator on $G(N,\psi_0;\cR)$ given by
\[
d:\sum_{n=0}^\infty\mathbf{a}_nq^n\mapsto\sum_{n=0}^\infty n\mathbf{a}_nq^n,
\]
and for each $m\in\bZ_{\geqslant 0}$ let $\binom{d}{m}$ denote the operator
given by $\sum_n\mathbf{a}_nq^n\mapsto\sum_n\binom{n}{m}\mathbf{a}_nq^n$.

\begin{defn}\label{def:mod-meas}
For any $\mathbf{f}\in G(N,\psi_0;{\cR})$ and $x\in\widehat{\rm Ig}(N)(\cR)$, let
$\mu_{\mathbf{f},x}$ be the $\cR$-valued measure on $\bZ_p$ determined by
\[
\int_{\bZ_p}\binom{z}{m}{\rm d}\mu_{\mathbf{f},x}(z)=\binom{d}{m}\mathbf{f}(x),
%\quad\textrm{for all $m\geqslant 0$.}
\]
for all $m\geqslant 0$.
\end{defn}

Setting $\mathbf{f}^\flat:=\sum_{(n,p)=1}\mathbf{a}_nq^n\in G(N,\psi_0;\cR)$, it is easy to
see that the associated measure $\mu_{\mathbf{f}^\flat,x}$ is then supported on $\bZ_p^\times$.

\subsection{CM points}\label{sec:CMpoints}

Let $K$ be an imaginary quadratic field of odd discriminant $-D_K<-3$, let $p>2$
be a prime split in $K$, and write
\[
p\cO_K=\pp\overline\pp,
\]
where $\pp$ is the prime of $K$ above $p$ induced by our fixed embedding
$\imath_p:\overline{\bQ}\hookrightarrow\bC_p$. We shall assume throughout that
$K$ satisfies the following \emph{Heegner hypothesis} relative to a fixed
integer $N>0$ prime to $p$:
\begin{equation}\label{HH}
\textrm{there is an ideal $\mathfrak{N}\subseteq\cO_K$ with $\cO_K/\mathfrak{N}\simeq\bZ/N\bZ$}.\tag{heeg}
\end{equation}
The existence of such $\mathfrak{N}$, which will be fixed from now on,
amounts to the requirement that every prime $q\mid N$ is either split in $K$
or it is ramified in $K$ with $q^2\nmid N$.
\sk

For each positive integer $c$ let $\cO_c=\bZ+c\cO_K$ be the order of $K$ of that conductor, and let
$H_c$ be the corresponding ring class field, so that ${\rm Gal}(H_c/K)\simeq{\rm Pic}(\cO_c)$ by
the Artin reciprocity map. %; in particular, $H:=K_1$ is the Hilbert class field of $K$.
%, whose Galois group is isomorphic to the Picard group ${\rm Pic}(\cO_K)$ via the Artin reciprocity map.
For each $\cO_c$-ideal $\fa$ prime to $\mathfrak{N}\pp$, let
$A_{\fa}/H_c$ be the CM elliptic curve with the complex uniformization $A_\fa(\bC)=\bC/\fa^{-1}$, and
equip $A_\fa$ with the $\Gamma_1(Np^\infty)$-level structure
\[
\eta_a:\boldsymbol{\mu}_N\oplus\boldsymbol{\mu}_{p^\infty}\hookrightarrow A_\fa[N]\oplus A_\fa[p^\infty]
\]
defined in \cite[p.6]{cas-hsieh1}, where $a\in\widehat{K}^\times$ is chosen so that $a\widehat{K}\cap\cO_c=\fa$.
The pair $(A_\fa,\eta_a)$ defines a point $x_\fa\in{\rm Ig}(N)$ defined over
a discrete valuation ring inside $\mathcal{V}:=\imath_p^{-1}(\cO_{\bC_p})\cap K^{\rm ab}$,
where $K^{\rm ab}$ is the maximal abelian extension of $K$ in $\overline{\bQ}$.
We let $x_c$ denote the point $x_\fa$ for $\fa=\cO_c$.
\sk

Write $c=c_op^n$ with $p\nmid c_o$ and $n\geqslant 0$, and decompose $c_o=c_o^+c_o^-$ with $c_o^+$ (resp. $c_o^-$)
only divisible by primes which are split (resp. nonsplit) in $K$. We similarly decompose $N=N^+N^-$,
and set $\mathfrak{C}^+:=c_o^+\cO_K$ and $\mathfrak{N}^+:=N^+\cO_K$.
Fix a square-root $\sqrt{-D_K}\in K$, and set
\[
\vartheta:=(D_K+\sqrt{-D_K})/2.
\]
Following \cite[\S{2.4}]{cas-hsieh1}, we
define the matrix $\varsigma^{(\infty)}=(\varsigma_q)\in{\rm GL}_2(\widehat{\bQ})$ by
\begin{itemize}
\item{} $\varsigma_q=1$, if $q\nmid c_o^+N^+p$,
\item{} $\varsigma_q=(\overline{\vartheta}-\vartheta)^{-1}
\left(\begin{matrix}\overline\vartheta&\vartheta\\1&1\end{matrix}\right)$, %\in{\rm GL}_2(K_{\qq})={\rm GL}_2(\bQ_q)$,
if $q=\qq\overline{\qq}$ with $\qq\mid\mathfrak{C}^+\mathfrak{N}^+\pp$,
%\item{} $\varsigma_p^{(n)}=\varsigma_p\left(\begin{matrix}p^n&0\\0&1\end{matrix}\right)$;
%\item{} $\varsigma_q=\left(\begin{matrix}0&1\\-1&0\end{matrix}\right)$, if $q\mid N^-$,
\end{itemize}
and the matrix $\gamma_{c}=(\gamma_{c,q})\in{\rm GL}_2(\widehat{\bQ})$ by
\begin{itemize}
\item{} $\gamma_{c,q}=1$, if $q\nmid cNp$,
\item{} $\gamma_{c,q}=
\left(\begin{matrix}q^{{\rm ord}_q(c)}&1\\0&1\end{matrix}\right)$, %\in{\rm GL}_2(K_{\qq})={\rm GL}_2(\bQ_q)$,
if $q=\qq\overline{\qq}$ with $\qq\mid\mathfrak{C}^+\mathfrak{N}^+\pp$,
\item{} $\gamma_{c,q}=
\left(\begin{matrix}1&0\\0&q^{{\rm ord}_q(c)-{\rm ord}_q(N)}\end{matrix}\right)$,
if $q\mid c_o^-N^-$,
\end{itemize}
and set $\xi_c:=\varsigma^{(\infty)}\gamma_c$. Under the complex uniformization
\[
[\cdot]:\mathfrak{H}\times{\rm GL}_2(\widehat\bQ)\longrightarrow{\rm Ig}(N)(\bC)
\]
deduced from $(\ref{eq:perfectoid})$ and the standard uniformization of $Y_1(Np^s)$,
the pair $(\vartheta,\xi_c)$ is sent to $x_{c}$. Moreover,
by Shimura's reciprocity law, %\cite[Cor.~4.20]{hida-Sh},
if $a\in\widehat{K}^{(p)\times}$ and $\fa=a\widehat{\cO}_c\cap K$
is the associated fractional ideal of $\cO_{c}$, then
\[
[(\vartheta,\overline{a}^{-1}\xi_c)]=x_\fa=x_{c}^{\sigma_\fa}\in{\rm Ig}(N)(H_{c}(\pp^\infty)),
\]
where $\sigma_\fa={\rm rec}_K(a^{-1})\vert_{H_{c}(\pp^\infty)}\in{\rm Gal}(H_{c}(\pp^\infty)/K)$
is the Artin symbol of $\fa$ over the compositum of $H_{c}$ with the ray class
field of $K$ of conductor $\pp^\infty$, and $a\mapsto\overline{a}$ denoted
the action of the nontrivial automorphism $\tau\in{\rm Gal}(K/\bQ)$ on $\mathbb{A}_K$.

\subsection{Anticyclotomic Hecke characters}

We say that a Hecke character $\psi:K^\times\backslash\mathbb{A}_K^\times\rightarrow\bC^\times$
has infinity type $(\ell_1,\ell_2)$, with $\ell_1, \ell_2\in\frac{1}{2}\bZ$ such that $\ell_1-\ell_2\in\bZ$,
if
\[
\psi_\infty(z)=z^{\ell_1-\ell_2}(z\overline{z})^{\ell_2},
%z^{\ell_1}\overline{z}^{\ell_2},
\]
where for each place $v$ of $K$, we let $\psi_v:K_v^\times\rightarrow\bC^\times$ be
the component of $\psi$ at $v$. The conductor of $\psi$ is the largest ideal $\mathfrak{c}\subseteq\cO_K$ such that
$\psi_\qq(u)=1$ for all $u\in(1+\mathfrak{c}\cO_{K,\qq})^\times\subseteq K_\qq^\times$.
If $\psi$ has conductor $\mathfrak{c}_\psi$ and
$\fa$ is any fractional ideal of $K$ prime to $\mathfrak{c}_\psi$,
we write $\psi(\fa)$ for $\psi(a)$, where $a\in\widehat{K}^\times$ is such that $a\widehat\cO_K\cap K=\fa$ and
$a_\qq=1$ for all $\qq$ dividing $\mathfrak{c}_\psi$.
As a function on fractional ideals, then $\psi$ satisfies
$\psi((\alpha))=\alpha^{\ell_2-\ell_1}(\alpha\overline{\alpha})^{-\ell_2}$
for all $\alpha\in K^\times$ with $\alpha\equiv 1\pmod{\mathfrak{c}_\psi}$.

\begin{defn}
If $\psi=\psi_{\rm fin}\psi_\infty$ is a Hecke character of $K$ with
has infinity type $(\ell_1,\ell_2)$, the \emph{$p$-adic avatar} of $\psi$ is the character
$\widehat{\psi}:K^\times\backslash\widehat{K}^\times\rightarrow\bC_p^\times$ defined by
\[
\widehat\psi(z)=\imath_p\imath_\infty^{-1}(\psi_{\rm fin}(z))z_\pp^{\ell_1}z_{\overline{\pp}}^{\ell_2}.
\]
Via the reciprocity map ${\rm rec}_K$, we shall often regard $\widehat\psi$
as a Galois character $\widehat\psi:G_K\rightarrow\bC_p^\times$.
\end{defn}

%For $c$ any positive integer prime to $p$,
Let $H_{p^\infty}=\bigcup_nH_{p^n}$
be the ring class field of $K$ of conductor $p^\infty$, and set
$\widetilde{\Gamma}:={\rm Gal}(H_{p^\infty}/K)$.
We say that a Hecke character $\psi:K^\times\backslash\mathbb{A}_K^\times\rightarrow\bC^\times$
is \emph{anticyclotomic} if $\psi\vert_{\mathbb{A}^\times}=\mathds{1}$.
The infinity type of such $\psi$ is of the form $(\ell,-\ell)$, %for some $\ell\in\frac{1}{2}\bZ$
and the correspondence $\psi\mapsto\widehat\psi$ establishes a bijection between the set of %(algebraic, as always implicit here)
anticyclotomic Hecke characters of $K$ of conductor dividing $p^\infty$ and the set of locally algebraic %$\overline{\bQ}_p$-valued
$\bC_p$-valued characters of $\widetilde{\Gamma}$.

\subsection{A two-variable anticyclotomic $p$-adic $L$-function}\label{subsec:Lp}

Throughout this section, we let $\mathbf{f}\in G(N,\psi_0;\cR)$ be an
$\cR$-adic modular form as in Definition~\ref{def:hida}. Also, let
$\lambda:K^\times\backslash\mathbb{A}_K^\times\rightarrow\ro^\times$ be a fixed
Hecke character of infinity type $(1,0)$ and conductor prime to $Np$.

\begin{defn}\label{def:critchar}
Let $\varepsilon_{\rm cyc}:G_{\bQ}\rightarrow\bZ_p^\times$ be the $p$-adic cyclotomic
character, and let $i\in\bZ/(p-1)\bZ$ be such that $\psi_0\vert_{(\bZ/p\bZ)^\times}=\omega^i$.
\begin{enumerate}
\item{} Define the \emph{critical character} $\Theta:G_\bQ\rightarrow\cR^\times$ by
\begin{equation}\label{def:crit}
\Theta(\sigma):=\omega^{i/2}(\sigma)\cdot[\langle\varepsilon_{\rm cyc}(\sigma)\rangle^{1/2}],\nonumber
\end{equation}
%for all $\sigma\in G_\bQ$,
where $\langle\cdot\rangle^{1/2}:\bZ_p^\times\rightarrow\Gamma^{\rm wt}$ is the composition
of the projection $\bZ_p^\times\twoheadrightarrow\Gamma^{\rm wt}$ with the map
$x\mapsto x^{1/2}$.
\item{} Define the $\cR$-adic character $\boldsymbol{\chi}:K^\times\backslash\mathbb{A}_K^\times\rightarrow\cR^\times$ by
\[
\boldsymbol{\chi}(x):=\Theta({\rm rec}_{\bQ}({\rm N}_{K/\bQ}(x))).
\]
%for all $x\in\mathbb{A}_K^\times$.
\item{} Denote by $\langle\lambda\rangle$ the composition of $\lambda$
with the projection onto the $\bZ_p$-free quotient of $\ro^\times$,
which then takes values in $\Gamma^{\rm wt}$, and
define $\boldsymbol{\xi}:K^\times\backslash\mathbb{A}_K^\times\rightarrow\cR^\times$ by
\begin{equation}\label{def:xi}
\bx(x)=\lambda^{1-\tau}(x)\cdot[\langle\lambda^{1-\tau}({x})\rangle^{1/2}],\nonumber
\end{equation}
where $\lambda^{1-\tau}(x):=\lambda(x)/\lambda(\overline{x})$.
\end{enumerate}
\end{defn}

\begin{rem}
Recall that we assume $p>2$ and note that implicit in Definition~\ref{def:critchar} is a
choice of a lift of $i$ to $\bZ/2(p-1)\bZ$; we fix either one of the two possible choices,
\emph{cf.} \cite[Rem.~2.1.3]{howard-invmath}.
\end{rem}

Assume that $c_o\cO_K$ is the conductor of $\lambda^{1-\tau}$, %(so in particular, $(c,pN)=1$),
and for any $\cO_{c_o}$-ideal $\fa$ prime to $\mathfrak{N}\pp$, let
$\mu_{\F^\flat,\mathfrak{a}}$ be the measure $\mu_{\F^\flat,x}$ of Definition~\ref{def:mod-meas}
associated to the CM point $x_\fa\in{\rm Ig}(N)$ of $\S\ref{sec:CMpoints}$.
Setting $T=t-1$, we shall denote by $\mu_{\F_\fa^\flat}$ the measure
on $\bZ_p^\times$ characterized by
\[
\mathscr{A}_{\mu_{\F_\fa^\flat}}(T)=\mathscr{A}_{\mu_{\F^\flat,\fa}}((1+T)^{\mathbf{N}(\fa)^{-1}\sqrt{-D_K}^{-1}}-1),
\]
%$\mathbf{f}_\fa^\flat(t)\in\cR[[t-1]]$
%the Amice transform $\mathscr{A}_{\mu_{\mathbf{f},\fa}}(T)$ of $\mu_{\mathbf{f},\fa}$,
and if $\phi:\bZ_p^\times\rightarrow\ro^\times$ is any continuous character,
define $\mathbf{f}_\fa^\flat\otimes\phi(t)\in\cR[[t-1]]$ by
\begin{align*}
\mathbf{f}_\fa^\flat\otimes\phi(t)&=\int_{\bZ_p}\phi(x)t^x{\rm d}\mu_{\mathbf{f}_\fa^\flat}(x)\\
&=\sum_{m\geqslant 0}\left[\int_{\bZ_p}\phi(x)\binom{x}{m}{\rm d}\mu_{\mathbf{f}_\fa^\flat}(x)\right](t-1)^m.
\end{align*}

\begin{defn}%[Two-variable anticyclotomic $p$-adic $L$-function]
\label{def:2varL}
The \emph{two-variable anticyclotomic $p$-adic $L$-function} attached to $\mathbf{f}$ and $\bx$ is
the $\cR$-valued measure $\mathscr{L}_{\pp,\bx}(\mathbf{f})$ on $\widetilde{\ac}$ given by
\[
\mathscr{L}_{\pp,\bx}(\mathbf{f})(\phi)=
\sum_{[\fa]\in{\rm Pic}(\cO_{c_o})}\bx\bchi^{-1}(\fa)\mathbf{N}(\fa)^{-1}
\cdot\left(\mathbf{f}_{\fa}^\flat\otimes\phi\vert[\fa]\right)(A_\fa,\eta_\fa),
\]
for all $\phi:\widetilde{\Gamma}\rightarrow\bC_p^\times$,
where $\phi\vert[\fa]$ is the character on $\bZ_p^\times$ defined by
$\phi\vert[\fa](z):=\phi({\rm rec}_K(a){\rm rec}_{{\pp}}(z))$.
\end{defn}

%As usual, we may view $\mathscr{L}_{\pp,\bx}(\mathbf{f})$ as an element in $\cR_{\unr}[[\widetilde{\ac}]]$,
%with $\unr$ a finite extension of $\widehat{\bZ}_p^{\rm ur}$ containing $\ro$.
%For every $\ro$-algebra homomorphism $\k:\cR\rightarrow\overline{\bQ}_p$
%we will let $\k(\mathscr{L}_{\pp,\bx}(\mathbf{f}))$
%denote the measure on $\widetilde\ac$ obtained by composing $\mathscr{L}_{\pp,\bx}(\mathbf{f})$ with
%the map $\cR_{\unr}\rightarrow\overline{\bQ}_p$ induced by $\k$, and similarly %for every $\k\in\mathcal{X}_{\ro}^a(\cR)$
%let $\Theta_\k$, $\chi_\k$ and $\xi_\k$ denote the compositions of $\Theta$, $\boldsymbol{\chi}$ and $\bx$ with $\k$, respectively.
%\sk

Now we describe the interpolation property satisfied by $\mathscr{L}_{\pp,\bx}(\mathbf{f})$.
For the statement, recall that if $f=\sum_{n=1}^\infty a_n(f)q^n$ is a normalized newform
of weight $k>0$ and $\psi$ is an anticyclotomic Hecke character of conductor $c\cO_K$,
the Rankin $L$-series $L(f/K,\psi,s)$ is defined by the analytic continuation
of the Dirichlet $L$-series defined by
\[
L(f/K,\psi,s)=\zeta(2s+1-k)\sum_{\fa}\frac{a_{{\rm N}(\fa)}(f)\psi(\fa)}{{\rm N}(\fa)^s},
%\quad\quad({\rm Re}(s)>\frac{k+1}{2}),
\]
%which converges for $s\in\bC$ with $\Re{s}>(k+1)/2$.
for $s\in\bC$ with ${\rm Re}(s)>\frac{k+1}{2}$,
where the sum is over the integral ideals $\fa$ of $K$ with $(\fa,c\cO_K)=1$.
In terms of automorphic $L$-functions, we have
\begin{equation}\label{eq:relation-L}
L(f/K,\psi,s)=L\biggl(s-\frac{k-1}{2},\pi_K\otimes\psi\biggr),
\end{equation}
where $\pi_K$ is the base change to $K$ of the automorphic representation of ${\rm GL}_2(\mathbb{A})$
generated by $f$. Thus since $\pi_K\otimes\psi$ is self-dual, $L(f/K,\psi,s)$ satisfies
a functional equation relating is values at $s$ and $k-s$.

For any $\ro$-algebra homomorphism $\nu_k:\cR\rightarrow\overline{\bQ}_p$ %of the form $\k_{k}$
with $k>0$ and $\psi$ an anticyclotomic Hecke character
of $K$ of conductor $c_op^n\cO_K$ with $p\nmid c_o$, define the $p$-adic multiplier $\mathscr{E}_\pp(f_\k,\psi)$ by
\[
\mathscr{E}_\pp(f_\k,\psi)=
\left\{
\begin{array}{ll}
\bigl(1-\frac{\k(\mathbf{a}_p)\psi_{\overline\pp}(p)}{p^{k/2}}\bigr)
\bigl(1-\frac{\psi_{\overline\pp}(p)\varepsilon_\nu(p)p^{k/2-1}}{\k(\mathbf{a}_p)}\bigr)&\textrm{if $n=0$;}\\
\frac{\varepsilon(\psi_\pp^{-1})}{p^{n}} & \textrm{if $n\geqslant 1$,}
\end{array}
\right.
\]
where $\varepsilon_\nu$ is the nebentypus of $f_\nu$, and set
\[
L^{\rm alg}(f_\nu/K,\psi,k/2):=\frac{\Gamma(k+\ell)\Gamma(\ell+1)}{(2\pi)^{k+2\ell+1}({\rm Im}\;\vartheta)^{k+2\ell}}
\cdot\frac{L(f_\k/K,\psi,k/2)}{\Omega_K^{2k+4\ell}},
\]
where $\Omega_K\in\bC^\times$ is a complex period attached to $K$ as in \cite[\S{2.5}]{cas-hsieh1}.

\begin{thm}\label{thm:bigLp}
%Let $\k\in\mathcal{X}_{\ro}^a(\cR)$ be an arithmetic prime of weight $2r\geqslant 2$ and trivial
Let $\nu=\nu_k$ for some $k>0$ and let $\widehat\phi$ be the $p$-adic avatar of an anticyclotomic Hecke character $\phi$ of $K$
of infinity type $(\ell,-\ell)$ with $\ell\geqslant 0$ and conductor $c_op^n\cO_K$ with $p\nmid c_o$.
Then:
\begin{align*}
\frac{\k(\mathscr{L}_{\pp,\bx}(\mathbf{f}))(\widehat\phi)^2}{\Omega_p^{2k+4\ell}}
&=L^{\rm alg}(f_\k/K,\xi_\k\phi,k/2)\cdot\mathscr{E}_\pp(f_\k,\x_\k\phi)^2\cdot
\phi(\mathfrak{N}^{-1})\cdot 2^3\cdot c_o\varepsilon(f_\k)\cdot w_K^2\sqrt{D_K},
\end{align*}
where $\varepsilon(f_\k)$ is the global root number of $f_\k$, $w_K:=\vert\cO_K^\times\vert$, and
$\Omega_p\in R_0^\times$ is a $p$-adic period as in \cite[\S{2.5}]{cas-hsieh1}.
\end{thm}

\begin{proof}
Let $\nu$ be as in the statement and set $f=f_\nu$. Then
\[
\Theta_\k(z)=z^{k/2-1}
\]
for all $z\in\bZ_p^\times$, and hence $\chi_\k(\fa)=\mathbf{N}(\fa)^{k/2-1}$. %and $\chi_{\k,\pp}$ is trivial.
Specializing $\mathscr{L}_{\pp,\bx}(\mathbf{f})$ at $\k$ we thus see that
\[
\k(\mathscr{L}_{\pp,\bx}(\mathbf{f}))(\widehat\phi)=\sum_{[\fa]\in{\rm Pic}(\cO_{c_o})}
\x_\k(\fa)\mathbf{N}(\fa)^{-r}\cdot\left(\widehat{f}_\fa^\flat\otimes\phi\vert[\fa]\right)(A_\fa,\eta_\fa).
\]
Since $\xi_\k$ is the $p$-adic avatar of an anticyclotomic
Hecke character of infinity type $(k/2,-k/2)$, the above shows that $\k(\mathscr{L}_{\pp,\bx}(\mathbf{f}))$
agrees with the $\unr$-valued measure $\mathscr{L}_{\pp,\xi_\k}(f)$ on $\widetilde{\Gamma}$ %\in\unr[[\widetilde{\Gamma}]]$
constructed in \cite[\S{3.3}]{cas-hsieh1}, so the result follows from [\emph{loc.cit.}, Prop.~3.8]. (Note
that in \cite{cas-hsieh1} only cusp form of even weights $k\geqslant 2$ are considered, but the construction of
$\mathscr{L}_{\pp,\xi_\k}(f)$ readily extends to any $k\in\mathbf{Z}_{\geqslant 1}$, and the results
quoted from \cite{hsieh} are available in this level of generality.)
\end{proof}

\begin{cor}\label{cor:2varL}
For every $\nu=\nu_k$ with $k>0$, the function $\k(\mathscr{L}_{\pp,\bx}(\mathbf{f}))$
is not identically zero.
\end{cor}

\begin{proof}
As shown in the proof of Theorem~\ref{thm:bigLp}, the specialization
$\k(\mathscr{L}_{\pp,\bx}(\mathbf{f}))$ %\in\unr[[\widetilde{\Gamma}]]$
agrees with the $p$-adic $L$-function
$\mathscr{L}_{\pp,\xi_\k}(f)$ constructed in \cite[\S{3.3}]{cas-hsieh1} with $f=f_\nu$, and so
the result similarly follows from [\emph{loc.cit.}, Thm.~3.9]. %similarly as in Theorem~\ref{thm:bigLp}.
\end{proof}

\section{Big logarithm maps}\label{sec:2varL}

In this section we construct a Perrin-Riou big logarithm map adapted to our global anticyclotomic setting.
Starting with \cite{PR115}, the cyclotomic theory of these maps has been widely studied in the literature; see e.g.
\cite{berger:Kato} and the references therein.
The construction we give here combines work of Ochiai \cite{Ochiai-Col} and Loeffler--Zerbes \cite{LZ2}.
%to which we refer the reader for some further details.

\subsection{Ochiai's map for nearly ordinary deformations}\label{sec:nord}

We keep the notations introduced in $\S\ref{subsec:crit}$ and $\S\ref{subsec:Lp}$;
in particular, $\ro$ denotes the ring of integers of finite extension of $L/\bQ_p$ and
$\cR$ is a finite flat extension of $\Lambda^{\rm wt}_{\ro}=\ro[[\Gamma^{\rm wt}]]$.
We also identify $G_{\bQ_p}:={\rm Gal}(\overline{\bQ}_p/\bQ_p)$ with
the decomposition group $D_p\subseteq G_{\bQ}$
determined by our fixed embedding $\imath_p:\overline{\bQ}\hookrightarrow\overline{\bQ}_p$.
%\sk

%The following definition is taken from \cite[\S{3}]{Ochiai-Col}.

\begin{defn}\label{def:ochiai}
Let $\bT$ be a free $\cR$-module of rank $2$ equipped with a continuous linear action of $G_\bQ$.
We say that $\bT$ is a \emph{$p$-ordinary deformation} if:
\begin{enumerate}
\item[(i)]{} the action of $G_\bQ$ on ${\rm det}(\bT)$ is given
by $\Theta^{-2}\varepsilon_{\rm cyc}^{-1}$;
\item[(ii)]{} there exists a filtration as $G_{\bQ_p}$-modules
\begin{equation}\label{eq:ordGr}
0\longrightarrow \fil^+\bT\longrightarrow\bT\longrightarrow \fil^-\bT\longrightarrow 0
\end{equation}
with $\fil^{\pm}\bT$ free of rank $1$ over $\cR$, and
with the action on $\fil^+\bT$ being unramified.
\end{enumerate}
\end{defn}

Fix a $p$-ordinary deformation $\bT$ as in Definition~\ref{def:ochiai}, and
let $\Psi:G_{\bQ_p}\rightarrow\cR^\times$ be the unramified character giving the action of $G_{\bQ_p}$ on $\fil^+\bT$.
%Also, denote by $\Psi_\nu$ the composition of $\Psi$
%with $\nu:\cR\rightarrow F_\nu$, and let ${\rm Fr}_p\in G_{\bQ_p}$ denote a geometric Frobenius element.
Let $\Gamma_{\rm cyc}$ %:={\rm Gal}(\bQ_{p,\infty}/\bQ_p)$
be the Galois group of the cyclotomic $\bZ_p$-extension of $\bQ_p$,
and let $\Lambda_{\rm cyc}$ be the free $\bZ_p[[\Gamma_{\rm cyc}]]$-module of rank one
where $G_{\bQ_p}$ acts via the inverse of the canonical character
$G_{\bQ_p}\twoheadrightarrow\Gamma_{\rm cyc}\hookrightarrow\bZ_p[[\Gamma_{\rm cyc}]]^\times$.

\begin{defn}
Set $\cI:=\cR\widehat{\otimes}\bZ_p[[\Gamma_{\rm cyc}]]$.
The \emph{nearly $p$-ordinary deformation} associated to $\bT$
is the $\cI$-module
\[
\cT:=\bT\widehat{\otimes}_{\bZ_p}\Lambda_{\rm cyc}
\]
equipped with the diagonal $G_{\bQ_p}$-action. From $(\ref{eq:ordGr})$, $\cT$
fits in an exact sequence of $\cR[[G_{\bQ_p}]]$-modules
\[
0\longrightarrow \fil^+\cT\longrightarrow\cT\longrightarrow \fil^-\cT\longrightarrow 0
\]
with $\fil^{\pm}\cT:=\fil^{\pm}\bT\widehat{\otimes}_{\bZ_p}\Lambda_{\rm cyc}$.
\end{defn}

Let $\epsilon:\Gamma_{\rm cyc}\simeq 1+p\bZ_p$ be the isomorphism
induced by the $p$-adic cyclotomic character. We denote by
$\mathcal{X}_{\ro}^a(\Gamma_{\rm cyc})$ the set of continuous characters
$\sigma:\Gamma_{\rm cyc}\rightarrow\overline{\bQ}_p^\times$
of the form $\sigma=\epsilon^{w_\sigma}\sigma_o$
for some integer $w_\sigma\geqslant 0$, called the \emph{weight} of $\sigma$,
and some finite order character $\sigma_o$. We then say that $\sigma$ has \emph{conductor} $p^r$
if so does $\sigma_o$ seen as a character on $\bZ_p^\times$.
\sk

Recall the set $\mathcal{X}_{\ro}^a(\cR)$ from Definition~\ref{def:arith-primes}, and
for every pair $(\k,\sigma)\in\mathcal{X}_{\ro}^a(\cR)\times\mathcal{X}_{\ro}^a(\Gamma_{\rm cyc})$
let $\cO_{\k,\sigma}$ be the extension
of $\cO_\k$ generated by the values of $\sigma$. With a slight abuse,
we shall also denote by $\cO_{\k,\sigma}$ the free $\cO_{\k,\sigma}$-module of rank one where $G_{\bQ_p}$ acts via
the character $\sigma$. Define
\begin{align*}
T_{\k,\sigma}:=\cT\otimes_{\mathcal{I},\nu}\cO_{\k,\sigma},&
\quad\quad
V_{\k,\sigma}:=T_{\k,\sigma}\otimes_{\bZ_p}\bQ_p,\\
\fil^\pm T_{\k,\sigma}:=\fil^\pm\cT\otimes_{\mathcal{I},\nu}\cO_{\nu,\sigma},&
\quad\quad
\fil^{\pm}V_{\k,\sigma}:=\fil^\pm T_{\k,\sigma}\otimes_{\bZ_p}\bQ_p,
\end{align*}
and for every finite extension $F/\bQ_p$, let
\[
{\rm Sp}_{\k,\sigma}^{}:H^1(F,\fil^+\cT)\longrightarrow H^1(F,\fil^+T_{\k,\sigma})\longrightarrow H^1(F,\fil^+V_{\k,\sigma})
\]
be the induced maps on cohomology.
\sk

For $V$ a finite-dimensional $L$-vector space with a continuous linear action of $G_F$,
we denote by $\mathbf{D}_{{\rm dR},F}(V)$ the filtered $(L\otimes_{\bQ_p}F)$-module
\[
\mathbf{D}_{{\rm dR},F}(V):=(V\otimes_{\bQ_p}\mathbf{B}_{\rm dR})^{G_F},
\]
where $\mathbf{B}_{\rm dR}$ is Fontaine's ring of $p$-adic de Rham periods.
If $V$ is a de Rham $G_F$-representation (i.e.,
${\rm dim}_F\mathbf{D}_{{\rm dR},F}(V)={\rm dim}_{L}V$),
then for any finite extension $E/F$ there is a canonical isomorphism
$D_{{\rm dR},E}(V)=E\otimes_{F}D_{{\rm dR},F}(V)$.
Denote by $\langle\;,\;\rangle$ the de Rham pairing
\[
\langle\;,\;\rangle:\mathbf{D}_{{\rm dR},F}(V)
\times\mathbf{D}_{{\rm dR},F}(V^*(1))\longrightarrow L\otimes_{\bQ_p}F\longrightarrow\bC_p,
\]
where $V^*={\rm Hom}_L(V,L)$. Let $\mathbf{B}_{\rm cris}\subseteq\mathbf{B}_{\rm dR}$ be the crystalline period ring
and define
\[
\mathbf{D}_{{\rm cris},F}(V):=(V\otimes_{\bQ_p}\mathbf{B}_{\rm cris})^{G_F};
\]
this is an $(L\otimes_{\bQ_p}F_0)$-module equipped with the action of
crystalline Frobenius $\Phi$, where $F_0$ is the maximal unramified subfield of $F$.
When $F=\bQ_p$, we write  $\mathbf{D}_{\rm dR}(V)=\mathbf{D}_{{\rm dR},\bQ_p}(V)$ and
$\mathbf{D}_{\rm cris}(V)=\mathbf{D}_{{\rm cris},\bQ_p}(V)$.
If $V$ is a crystalline representation (i.e., ${\rm dim}_{F_0}\mathbf{D}_{{\rm cris}, F}(V)={\rm dim}_{L}V$),
then we have a canonical isomorphism $F\otimes_{F_0}\mathbf{D}_{{\rm cris},F}(V)=
\mathbf{D}_{{\rm dR},F}(V)$. Suppose further that
\[
\mathbf{D}_{{\rm cris},F}(V)^{\Phi=1}=\{0\}.
\]
Then we denote by $\log$ the Bloch--Kato logarithm map
%\begin{equation}\label{def:log}
\[
\log:=\log_{F,V} :H^1_{\rm f}(F,V)\longrightarrow
\frac{\mathbf{D}_{{\rm dR},F}(V)}{{\rm Fil}^0\mathbf{D}_{{\rm dR},F}(V)}
={\rm Fil}^0\mathbf{D}_{{\rm dR},F}(V^*(1))^\vee,
%\end{equation}
\]
where $H^1_{\rm f}(F,V)\subseteq H^1(F,V)$ is the Bloch--Kato subspace \cite[(3.7.2)]{BK},
and denote by $\exp^*$ the dual exponential map
\[
\exp^*:=\exp^*_{F,V}:H^1(F,V^*(1))\longrightarrow {\rm Fil}^0\mathbf{D}_{{\rm dR},F}(V^*(1))
\]
obtained by dualizing the Bloch--Kato exponential map
\[
{\rm exp}:={\rm exp}_{F,V}\colon
\frac{\mathbf{D}_{{\rm dR},F}(V)}{{\rm Fil}^0\mathbf{D}_{{\rm dR},F}(V)}
\longrightarrow H^1_{}(F,V)
\]
with respect to the de Rham and local Tate pairings (\emph{cf.}~\cite[\S 2.4]{LZ2}).

\begin{defn}
Let $\bT$ be a $p$-ordinary deformation, and set
\begin{equation}\label{def:D}
\mathbb{D}:=(\fil^+\bT\widehat{\otimes}_{\bZ_p}
\widehat{\bZ}_p^{\rm nr})^{G_{\bQ_p}},
\end{equation}
where the $G_{\bQ_p}$-action on $\fil^+\bT\widehat{\otimes}_{\bZ_p}
\widehat{\bZ}_p^{\rm nr}$ is the diagonal one. Also set
\[
\cD:=\mathbb{D}\widehat{\otimes}_{\bZ_p}\bZ_p[[\Gamma_{\rm cyc}]].
\]
\end{defn}

Fix a compatible system $(\zeta_{p^r})_r$ of $p$-power roots of unity.
Then as in \cite[Def.~3.12]{Ochiai-Col},
for every $(\k,\sigma)\in\mathcal{X}_{\ro}^a(\cR)\times\mathcal{X}_{\ro}^a(\Gamma_{\rm cyc})$
there are specialization maps
\begin{equation}\label{def:D}
{\rm Sp}^{}_{\k,\sigma}:
\cD\longrightarrow\mathbf{D}_{\rm dR}^{}(\fil^+V_{\k,\sigma}).
%=(\fil^+V_{\k,\sigma}\otimes_{\bQ_p}B_{\rm dR})^{G_{\bQ_p}}
\end{equation}
%where $D_{\rm dR}^{}(\fil^+V_{\k,\sigma})$
%is Fontaine's de Rham Dieudonn\'e module.

\begin{thm}\label{thm:Exp}
Let $\gamma_o\in\Gamma_{\rm cyc}$ be a topological generator,
and define
\[
\mathcal{J}:=(\Psi({\rm Fr}_p)-1,\gamma_o-1)\subseteq\cI.
\]
For any finite unramified extension $F/\bQ_p$ with ring of integers $\cO_F$ there
exists an injective $\cI$-linear map
\[
\mathcal{E}^{\Gamma_{\rm cyc}}_{F}:\mathcal{J}(\mathcal{D}\otimes_{\bZ_p}\cO_F)
\longrightarrow H^1(F,\fil^+\cT)
\]
such that for every $\k\in\mathcal{X}_{\ro}^a(\cR)$ of weight $k\geqslant 2$
and $\sigma\in\mathcal{X}_{\ro}^a(\Gamma_{\rm cyc})$ of weight $w$ with $1\leqslant w\leqslant k-1$
and conductor $p^n$, the following diagram commutes:
\begin{equation}\label{eq:Expdiag}
\xymatrix{
\mathcal{J}(\mathcal{D}\otimes_{\bZ_p}\cO_F)\ar[rr]^-{\mathcal{E}^{\Gamma_{\rm cyc}}_{F}}
\ar[d]^-{{\rm Sp}^{}_{\k,\sigma}} &&  H^1(F,\fil^+{\cT})
\ar[d]^-{{\rm Sp}_{\k,\sigma}^{}} \nonumber\\
\mathbf{D}_{{\rm dR},F}(\fil^+{V}_{\k,\sigma}^{})\ar[rr]^-{} && H^1(F,\fil^+{V}_{\k,\sigma}),
}
\end{equation}
where the bottom horizontal map is given by
\[
%\mathfrak{g}(\sigma_o^{-1})\cdot\left(\frac{p^{w-1}}{\Psi_\nu({\rm Fr}_p)}\right)^{r}
%\cdot\left(1-\frac{p^{w-1}\sigma_o(p)}{\Psi_\nu({\rm Fr}_p)}\right)
%\left(1-\frac{\Psi_\nu({\rm Fr}_p)\sigma_o(p)}{p^w}\right)^{-1}
%\cdot(-1)^{w-1}(w-1)!\cdot{\rm exp}_{F},
(-1)^{w-1}(w-1)!\cdot{\rm exp}_{}\times
\left\{
\begin{array}{ll}
\bigl(1-\frac{p^{w-1}}{\Psi_\k({\rm Fr}_p)}\bigr)
\bigl(1-\frac{\Psi_\k({\rm Fr}_p)}{p^w}\bigr)^{-1}&\textrm{if $n=0$;}\\
\mathfrak{g}(\sigma_o^{-1})\bigl(\frac{p^{w-1}}{\Psi_\k({\rm Fr}_p)}\bigr)^{r}&\textrm{if $n\geqslant 1$.}
\end{array}
\right.
\]
%where ${\rm exp}_F$ is the Bloch--Kato exponential map over $F$.
\end{thm}

\begin{proof}
See \cite[Prop.~5.3]{Ochiai-Col}.
\end{proof}

\subsection{Going up the unramified $\bZ_p$-extension}
\label{sec:unr}

Let $U:={\rm Gal}(F_\infty/\bQ_p)$ be the Galois group of the unramified $\bZ_p$-extension of $\bQ_p$,
let $F_n$ be the subfield of $F_\infty$ with
${\rm Gal}(F_n/\bQ_p)\simeq\bZ/p^n\bZ$, and set $U_n:={\rm Gal}(F_\infty/F_n)$.
Let $y_n:\cO_{F_n}\rightarrow\cO_{F_n}[U/U_n]$ be
the $\bZ_p$-linear map defined by
\[
y_n(x)=\sum_{\sigma\in U/U_n}x^\sigma[\sigma^{-1}],
\]
and let $\mathcal{S}_n$ be the image of $y_n$.
\sk

For any $x\in\cO_{F_{n+1}}$, it is readily seen that the image of $y_{n+1}(x)$ in
$\cO_{F_{n+1}}[U/U_n]$ agrees with $y_n({\rm Tr}_{F_{n+1}/F_n}(x))$, and hence
passing to the inverse limits with respect to the trace maps, we obtain an isomorphism
\begin{equation}\label{eq:yager}
\varprojlim_ny_n\colon\varprojlim_n\cO_{F_n}\xrightarrow{\;\simeq\;}\cS_\infty:=\varprojlim_n\cS_n.
\end{equation}

%Following \cite[Def.~3.7]{LZ2}, we refer to $\cS_\infty$ as the \emph{Yager module}.

\begin{prop}\label{lem:Yager}
The module $\cS_\infty$ is free of rank $1$ over $\bZ_p[[U]]$,
and it is identified with
\[
%\cS_\infty=
\{g\in\widehat{\cO}_{F_\infty}[[U]]\;\colon\;g^u=[u]g\;\textrm{ for all $u\in U$}\},
\]
where $\widehat{\cO}_{F_\infty}$ is
the completion of the ring of integers of $F_\infty$.
\end{prop}

\begin{proof}
See \cite[Prop.~3.2, Prop.~3.6]{LZ2}.
\end{proof}

%\begin{cor}
%The module $\cD_\infty:=\cD\hat{\otimes}_{\bZ_p}\cS_\infty$
%is free of rank $1$ over $\cI_\infty:=\cI[[U]]$.
%\end{cor}

%\begin{proof}
%By \cite[Lemma~3.3]{Ochiai-Col}, the module $\cD$ is free of rank $1$ over $\cI$.
%The result thus follows from Lemma~\ref{lem:Yager}.
%\end{proof}

\subsection{A two-variable regulator map for ordinary deformations}

Let $G={\rm Gal}(L_\infty/\bQ_p)$ be the Galois group of the unique $\bZ_p^2$-extension of $\bQ_p$,
and note that $L_\infty=\bQ_{p,\infty}F_\infty$. %, and $G\cong\Gamma_{\rm cyc}\times U$ canonically.
As in $\S\ref{sec:nord}$, we let $\bT$ be a $p$-ordinary deformation in the sense of Definition~\ref{def:hida},
and let $\Psi:G_{\bQ_p}\rightarrow\cR^\times$ be the unramified character giving the action
on the unramified $\cR$-line $\mathscr{F}^+\bT\subseteq\bT$.

\begin{defn}
%Let $\bT\cong\cR^2$ be a nearly ordinary deformation.
An arithmetic prime $\k\in\mathcal{X}_{\ro}^a(\cR)$ is \emph{exceptional} for $\bT$
if $\k=\k_2$ %has weight $2$, the wild character $\psi_\k$ is trivial,
and $\Psi_\k({\rm Fr}_p)=1$.
\end{defn}

For any subquotient $\mathbb{M}$ of $\bT$ define
\[
H^1_{\rm Iw}(L_\infty,\mathbb{M}):=\varprojlim_LH^1(L,\mathbb{M}),
\]
where the limit is over the finite extensions $L/\bQ_p$ contained in $L_\infty$
with respect to the corestriction maps.

\begin{thm}\label{thm:Log}
Let $\lambda:=\Psi({\rm Fr}_p)-1$.
There is an injective $\cR[[G]]$-linear map
\[
\mathcal{L}^{G}:H_{\rm Iw}^1(L_\infty,\fil^+{\bT})
\longrightarrow\lambda^{-1}\cdot\cJ(\mathbb{D}\widehat{\otimes}_{\bZ_p}\widehat{\cO}_{F_\infty}[[G]])\nonumber
\]
such that for every $\mathfrak{Y}_\infty\in H_{\rm Iw}^1(L_\infty,\fil^+{\bT})$,
and for every non-exceptional $\k\in\mathcal{X}^a_{\ro}(\cR)$ of weight $k\geqslant 2$ and
Hodge--Tate character $\phi:G\rightarrow L^\times$ of conductor $p^n$ and Hodge--Tate
weight\footnote{In this paper, we adopt the convention that the Hodge--Tate weight
of $\varepsilon_{\rm cyc}$ is $+1$.
Thus the Hodge--Tate weights of a $p$-adic de Rham representation $V$ are the integers $w$
such that ${\rm Fil}^{-w}\mathbf{D}_{\rm dR}(V)\supsetneq{\rm Fil}^{-w+1}\mathbf{D}_{\rm dR}(V)$.}
$w$, with $1\leqslant w\leqslant k-1$, we have
\begin{align*}
\mathcal{L}^{G}(\mathfrak{Y}_\infty)(\k,\phi)=
%\frac{\Psi_\k({\rm Fr}_p)^r}{\varepsilon(\phi)}\cdot
%\frac{{\mathscr{P}}^*(\k,\phi^{-1})}{{\mathscr{P}}(\k,\phi)}\cdot
\frac{(-1)^{w-1}}{(w-1)!}&\cdot{\rm log}(\k(\mathfrak{Y}_\infty)^{\phi^{-1}})
%&\times
\times\left\{
\begin{array}{ll}
\bigl(1-\frac{\Psi_\k({\rm Fr}_p)}{p^w}\bigr)
\bigl(1-\frac{p^{w-1}}{\Psi_\k({\rm Fr}_p)}\bigr)^{-1} & \textrm{if $n=0$;}\\
\varepsilon(\phi)^{-1}\Psi_\k({\rm Fr}_p^n) & \textrm{if $n\geqslant 1$,}\\
\end{array}
\right.
\end{align*}
where %$\phi_o:=\phi\varepsilon_{\rm cyc}^{-w}$ and
$\varepsilon(\phi)$ is the $\varepsilon$-factor of $\phi$.
%${\mathscr{P}}^*(\k,\phi^{-1})=(1-\frac{\Psi_\k({\rm Fr}_p)}{p^w}\phi_o^{-1}({\rm Fr}_p))$
%and ${\mathscr{P}}(\k,\phi)=(1-\frac{p^{w-1}}{\Psi_\k({\rm Fr}_p)}\phi_o({\rm Fr}_p))$ with if $r=0$,
%${\mathscr{P}}^*(\k,\phi^{-1})={\mathscr{P}}(\k,\phi)=1$ if $r\geqslant 1$,
%and ${\rm log}$ is the Bloch--Kato logarithm map.
\end{thm}

\begin{proof}
For each $n\geqslant 0$, let
\[
\mathcal{E}^{\Gamma_{\rm cyc}}_{F_n}:
\mathcal{J}(\cD\otimes_{\bZ_p}\cO_{F_n})\longrightarrow H^1(F_n,\fil^+{\cT})
\]
be the big exponential map of Theorem~\ref{thm:Exp} for the unramified extension $F_n/\bQ_p$,
and using $(\ref{eq:yager})$ define
\[
\mathcal{E}^{G}:=
\varprojlim_n\mathcal{E}^{\Gamma_{\rm cyc}}_{F_n}:
\mathcal{J}(\mathcal{D}\widehat{\otimes}_{\bZ_p}\cS_\infty)\longrightarrow H^1_{\rm Iw}(F_\infty,\fil^+{\cT}).
\]
By Shapiro's lemma, we view $\mathcal{E}^G$ as taking values
in $H^1_{\rm Iw}(L_\infty,\fil^+{\bT})$.
%
%Since each of the maps ${\rm Exp}_{\fil_w^+(\cT)}^{(n)}$ is injective with cokernel killed by $\lambda$ and the transition maps in H_Iw are surj by def,
Since each $\mathcal{E}^{\Gamma_{\rm cyc}}_{F_n}$ has cokernel killed by $\lambda$,
it is readily seen that $\mathcal{E}^G$ is an injective $\cR[[G]]$-linear map with cokernel
killed by $\lambda$, and hence given any $\mathfrak{Y}_\infty\in H_{\rm Iw}^1(L_\infty,\fil^+{\bT})$,
the assignment
\[
\mathcal{L}^G(\mathfrak{Y}_\infty):=\lambda^{-1}
\cdot(\mathcal{E}^G)^{-1}(\lambda\cdot\mathfrak{Y}_\infty)
\]
is a well-defined element in
\[
\lambda^{-1}\cdot\mathcal{J}(\mathcal{D}\widehat{\otimes}_{\bZ_p}\cS_\infty)
\hookrightarrow\lambda^{-1}\cdot\mathcal{J}(\mathcal{D}\widehat{\otimes}_{\bZ_p}\widehat{\cO}_{F_\infty}[[U]])
\simeq\lambda^{-1}\cdot\mathcal{J}(\mathbb{D}\widehat{\otimes}_{\bZ_p}\widehat{\cO}_{F_\infty}[[G]]).
\]
Thus constructed, the interpolation properties of $\mathcal{L}^G$
for each non-exceptional $\k\in\mathcal{X}_{\ro}^a(\cR)$
then follow as in \cite[Thm.~4.15]{LZ2}.
\end{proof}

\begin{defn}
Let $\mathbf{f}\in\cR[[q]]$ be an ordinary $\cR$-adic newform of tame level $N$ (prime to $p$).
We say that an arithmetic prime $\k\in\mathcal{X}_{\ro}^a(\cR)$ is \emph{$p$-old}
if $\F_\k$ is the $p$-stabilization of an ordinary newform of level $N$.
\end{defn}

Note that if $\k\in\mathcal{X}_{\ro}^a(\cR)$ has weight $k>2$ and trivial nebentypus, then $\k$ is $p$-old
(see \cite[Lemma~2.1.5]{howard-invmath}), and that any $p$-old arithmetic prime is also
non-exceptional.

\begin{cor}\label{thm:ERL}
%Let $\mathcal{L}^G$ be the two-variable regulator map of Theorem~\ref{thm:Log},
 %and let $\mathfrak{Y}_\infty$ be a class in $H_{\rm Iw}^1(L_\infty,\fil^+({\bT}))$.
Let $\k\in\mathcal{X}_{\ro}^a(\cR)$ be a $p$-old arithmetic prime.
Then for every Hodge--Tate character $\phi:G\rightarrow L^\times$ of Hodge--Tate weight
$w\leqslant 0$ and conductor $p^n$ we have
\begin{align*}
\mathcal{L}^{G}(\mathfrak{Y}_\infty)(\k,\phi)=
%\frac{\Psi_\k({\rm Fr}_p)^r}{\varepsilon(\phi)}\cdot
%\frac{{\mathscr{P}}^*(\k,\phi^{-1})}{{\mathscr{P}}(\k,\phi)}\cdot
(-w)!&\cdot{\rm exp^*}(\k(\mathfrak{Y}_\infty)^{\phi^{-1}})
%&\times
\times\left\{
\begin{array}{ll}
\bigl(1-\frac{\Psi_\k({\rm Fr}_p)}{p^w}\bigr)
\bigl(1-\frac{p^{w-1}}{\Psi_\k({\rm Fr}_p)}\bigr)^{-1} & \textrm{if $n=0$;}\\
\varepsilon(\phi)\Psi_\k({\rm Fr}_p^n) & \textrm{if $n\geqslant 1$.}\\
\end{array}
\right.
\end{align*}
\end{cor}

\begin{proof}
The specialization of $\mathcal{L}^G$ at $\k$ gives rise to
an $\cO_\k[[G]]$-linear map
\[
\mathcal{L}_{\k}^G:H^1_{\rm Iw}(L_\infty,\fil^+{V}_\k)
\longrightarrow\mathbf{D}_{\rm dR}(\fil^+{V}_\k)\otimes_{\bZ_p}\widehat{\cO}_{F_\infty}[[G]]
\]
which by Theorem~\ref{thm:Log} enjoys the same interpolation properties
at a dense set characters of $G$ as the map $\mathcal{L}_{V}^G$ constructed
in \cite[Thm.~4.7]{LZ2} for $V=\fil^+{V}_\k$.
(Note that since $\k$ is $p$-old, $\fil^+V_\k$
is indeed a crystalline $G_{\bQ_p}$-representation.)
Since $\mathcal{L}_{\fil^+{V}_\k}^G$ is uniquely determined by its values
at such characters (for every given class in $H^1_{\rm Iw}(L_\infty,\fil^+{V}_\k)$),
the result follows from \cite[Thm.~4.15]{LZ2}.
\end{proof}

\section{Big Heegner points}
\label{sec:HP}

Let $f\in S_k(\Gamma_0(N))$ be a $p$-ordinary newform of level $N$ prime to $p>3$,
and let $K/\bQ$ be an imaginary quadratic field as in $\S\ref{sec:CMpoints}$;
in particular, $K$ satisfies condition (heeg) relative to $N$. Let $L/\bQ_p$ be a finite extension
with ring of integers $\ro$ containing the Fourier coefficients of $f$.
In this section, we briefly recall Howard's construction
of big Heegner points associated to the ordinary $\cR$-adic newform passing through $f$.
%For simplicity, we assume $k\equiv 2\pmod{p-1}$ and that $p$ does not divide the class number of $K$.

\subsection{Galois representations associated to Hida families}\label{subsec:Gal-Hida}

Let ${X_s}_{/\bQ}$ be the compactified modular curve whose non-cuspidal points classify
isomorphism classes of triples $(E,C,\pi)$ with:
\begin{itemize}
\item{} $E_{}$ is an elliptic curve over an arbitrary $\bQ$-scheme $S$;
\item{} $C$ is a cyclic subgroup of $E$ of order $N$;
\item{} $\pi$ is a point of $E$ of exact order $p^s$.
\end{itemize}

Let $J_s:={\rm Pic}^0(X_s)\otimes_{\bZ}\ro$ be the Jacobian of $X_s$, denote by $\mathfrak{h}_s$ the $\ro$-algebra
generated by the Hecke operators $T_\ell$ ($\ell\nmid Np$), $U_\ell$ ($\ell\mid Np$), and $\langle a\rangle$
($a\in(\bZ/N\bZ)^\times$) acting of $J_s$ by Albanese functoriality, and let
\[
e^{\rm ord}:=\lim_{m\to\infty}U_p^{m!}
\]
be Hida's ordinary projector. By \cite[Thm.~1.1]{hida86b}, the algebra
$\mathfrak{h}^{\rm ord}:=\varprojlim_se^{\rm ord}\mathfrak{h}_s$ is finite flat over $\Lambda_{\ro}^{\rm wt}$.
Let $\mathfrak{h}^{\rm ord}_{\mathfrak{m}}$ be the local summand of $\mathfrak{h}^{\rm ord}$ through which the algebra homomorphism
$\lambda_f:\mathfrak{h}^{\rm ord}\rightarrow\ro$ defined by $f$ factors, let
$\mathfrak{a}\subseteq\mathfrak{h}^{\rm ord}_{\mathfrak{m}}$ be the unique minimal prime
containing the kernel of $\lambda_f$, and set
\[
\mathbb{I}:=\mathfrak{h}^{\rm ord}_{\mathfrak{m}}/\mathfrak{a}.
\]
Letting $\mathbf{a}_n\in\cR$ be the image of $T_n\in\mathfrak{h}^{\rm ord}$,
the formal $q$-expansion $\mathbf{f}=\sum_{n\geqslant 1}\mathbf{a}_nq^n\in\cR[[q]]$
is an ordinary $\cR$-adic newform of tame level $N$ and character $\omega^{k-2}$
in the sense of Definition~\ref{def:hida}.
\sk

Let $\kappa_L$ be the residue field of $L$ and denote by
$\bar{\rho}_f:G_\bQ\rightarrow{\rm GL}_2(\kappa_L)$ the semisimple
residual representation attached to $f$.

\begin{thm}\label{thm:MT}
Assume that %the residual Galois representation
$\bar{\rho}_f$ is irreducible and $p$-distinguished. Then the following hold:
\begin{itemize}
\item{}
The module
\[
\mathbf{T}:=\biggl(\varprojlim_se^{\rm ord}({\rm Ta}_p(J_s)\otimes_{\bZ_p}\ro)\biggr)\otimes_{\mathfrak{h}^{\rm ord}}\cR
\]
is free of rank $2$ over $\cR$.
\item{} The Galois representation
\[
\rho_{\F}:G_\bQ\longrightarrow{\rm Aut}_\cR(\mathbf{T})\simeq{\rm GL}_2(\cR)
\]
is unramified outside $Np$ with
\[
{\rm trace}(\rho_{\F})({\rm Fr}^{-1}_\ell)=\mathbf{a}_\ell,
\quad\quad{\rm det}(\rho_{\F})({\rm Fr}^{-1}_\ell)=\ell[\ell],
\]
for all $\ell\nmid Np$, where ${\rm Fr}^{-1}_\ell$ is an arithmetic Frobenius.
\item{}
As a representation of $G_{\bQ_p}$, there is an exact sequence
\begin{equation}\label{eq:ordGr2}
0\longrightarrow\fil^+\mathbf{T}\longrightarrow\mathbf{T}\longrightarrow\fil^-\mathbf{T}\longrightarrow 0
\end{equation}
with $\fil^{\pm}\mathbf{T}\simeq\cR$, and with the action of $G_{\bQ_p}$
on $\fil^-\mathbf{T}$ given by the unramified character $\alpha:G_{\bQ_p}\rightarrow\cR^\times$
sending ${\rm Fr}_p^{-1}$ to $\mathbf{a}_p$.
\end{itemize}
\end{thm}

\begin{proof}
%Since we assume that $\bar\rho_f$ is absolutely irreducible and $p$-distinguished, the result
This follows from \cite[Thm.~7]{Mazur-Tilouine} %for $(i)$,
and \cite[Thm.~2.2.2]{wiles88}. %for $(ii)$.
\end{proof}

\subsection{Howard's big Heegner points}
\label{subsec:bigHP}

Fix a positive integer $c$ prime to $Np$.
The CM points $x_{cp^n}\in{\rm Ig}(N)(\bC)$ constructed in $\S\ref{sec:CMpoints}$ descend to points
$P_{cp^n,s}\in X_s(H_{cp^n}(\boldsymbol{\mu}_{p^s}))$, for all $n\geqslant s$.

\begin{lem}\label{lem:points}\hfill
%For varying $n\geqslant s$, the points $P_{p^n,s}$ satisfy the following properties:
\begin{enumerate}
\item{} For all $\sigma\in{\rm Gal}(H_{cp^n}(\boldsymbol{\mu}_{p^s})/H_{cp^n})$, we have
\[
P_{cp^n,s}^\sigma=\langle\vartheta(\sigma)\rangle\cdot P_{cp^n,s},
\]
where $\vartheta:{\rm Gal}(H_{cp^n}(\boldsymbol{\mu}_{p^s})/H_{cp^{n}})\rightarrow\bZ_p^\times/\{\pm{1}\}$ is such that
$\vartheta^2=\varepsilon_{\rm cyc}$.
\item{} If $n\geqslant s> 1$, then
\[
\sum_{\sigma\in{\rm Gal}(H_{cp^n}(\boldsymbol{\mu}_{p^s})/H_{cp^{n-1}}(\boldsymbol{\mu}_{p^s}))}
\alpha_s(P_{cp^n,s}^{{\sigma}})
=U_p\cdot P_{cp^n,s-1},
\]
where $\alpha_s:X_s\rightarrow X_{s-1}$ is the map given by $(E,C,\pi)\mapsto (E,C,p\cdot\pi)$
on non-cuspidal moduli.
\item{} If $n\geqslant s\geqslant 1$, then
\[
\sum_{\sigma\in{\rm Gal}(H_{cp^n}(\boldsymbol{\mu}_{p^s})/H_{cp^{n-1}}(\boldsymbol{\mu}_{p^s}))}
P_{cp^n,s}^{{\sigma}}
=U_p\cdot P_{cp^{n-1},s}.
\]
\end{enumerate}
\end{lem}

\begin{proof}
From the construction of $x_{cp^n}$, %given in $\S\ref{sec:CMpoints}$,
it is immediate to see
that the point $P_{cp^n,s}$ for $n\geqslant s$ agrees with the point $h_{cp^{n-s},s}\in X_s(\bC)$
corresponding to the triple $(A_{cp^{n-s},s},\nn_{cp^{n-s},s},\pi_{cp^{n-s},s})$ with:
\begin{itemize}
\item{} $A_{cp^{n-s},s}(\bC)=\bC/\cO_{cp^{n}}$;
\item{} $\nn_{cp^{n-s},s}=A_{cp^{n-s},s}[\mathfrak{N}\cap\cO_{cp^{n}}]$;
\item{} $\pi_{cp^{n-s},s}$ a generator of the kernel of the cyclic $p^s$-isogeny
$\bC/\cO_{cp^{n}}\rightarrow\bC/\cO_{cp^{n-s}}$,
\end{itemize}
as constructed in \cite[\S{2.2}]{howard-invmath}.
The proof of properties (1), (2) and (3) %in the Lemma
thus follows from (the proof of) Corollary~2.2.2, Lemma~2.2.4 and Proposition~2.3.1 of \cite{howard-invmath},
respectively.
\end{proof}

Set $L_{c,s}:=H_{cp^s}(\boldsymbol{\mu}_{p^s})$, and let $e_i$ denote the
idempotent of $\bZ_p[[\bZ_p^\times]]$ projecting to the $\omega^i$-th isotypical component.
As in \cite[Cor.~2.2.2]{howard-invmath}, it follows easily from
Lemma~\ref{lem:points} that the points $e_{k-2}e^{\rm ord}P_{cp^{t+1+s},s}$
define classes
\[
y_{cp^{t+1},s}\in e^{\rm ord}J_s(L_{cp^{t+1},s})
\]
which satisfy
\begin{equation}\label{eq:gal}
y_{cp^{t+1},s}^\sigma=\Theta(\sigma)\cdot y_{cp^{t+1},s},
\quad\textrm{for all $\sigma\in{\rm Gal}(L_{cp^{t+1},s}/H_{cp^{t+1+s}})$},
\end{equation}
where $\Theta$ is the character defined in $(\ref{def:crit})$,
viewed as acting on $J_s$ via the diamond operators.

\begin{defn}
For any $\Lambda_{\ro}^{\rm wt}$-module $M$ equipped with a linear $G_{\bQ}$-action,
we let $M^\dagger$ denote its twist by the character $\Theta^{-1}$.
\end{defn}

Thus by (\ref{eq:gal}) we have
\[
y_{cp^{t+1},s}\in H^0(H_{cp^{t+1+s}},e^{\rm ord}J_{s}(L_{cp^{t+1},s})^\dagger).
\]

For any $m>0$, let $\mathfrak{G}_{H_m}$ be the Galois group of the maximal extension of $H_{m}$
unramified outside the primes above $Np$. By Lemma~\ref{lem:points},
the image of $y_{cp^{t+1},s}$ under the composite map
\begin{align*}
H^0(H_{cp^{t+1+s}},e^{\rm ord}J_{s}^{\rm ord}(L_{cp^{t+1},s})^\dagger)
&\xrightarrow{{\rm Cor}}H^0(H_{cp^{t+1}},e^{\rm ord}J_{s}(L_{cp^{t+1},s})^\dagger)\\
&\xrightarrow{{\rm Kum}}H^1(\mathfrak{G}_{H_{cp^{t+1}}},e^{\rm ord}{\rm Ta}_p(J_s)^\dagger)
\end{align*}
defines a class $\mathfrak{X}_{cp^{t+1},s}$ satisfying
\[
\alpha_{s*}\mathfrak{X}_{cp^{t+1},s}=U_p\cdot\mathfrak{X}_{cp^{t+1},s-1}
\]
under the map
\[
\alpha_{s*}:H^1(\mathfrak{G}_{H_{cp^{t+1}}},e^{\rm ord}{\rm Ta}_p^{}(J_{s})^\dagger)
\longrightarrow H^1(\mathfrak{G}_{H_{cp^{t+1}}},e^{\rm ord}{\rm Ta}_p^{}(J_{s-1})^\dagger)
\]
induced by $\alpha_s:X_s\rightarrow X_{s-1}$ by Albanese functoriality.

\begin{defn}%[Howard]
The \emph{big Heegner point of conductor $cp^{t+1}$} is the class
\[
\mathfrak{X}_{cp^{t+1}}\in H^1(\mathfrak{G}_{H_{cp^{t+1}}},\mathbf{T}^\dagger)
\]
defined as the image of $\varprojlim_sU_p^{-s}\cdot\mathfrak{X}_{cp^{t+1},s}$
under the natural map
\[
\varprojlim_sH^1(\mathfrak{G}_{H_{cp^{t+1}}},e^{\rm ord}{\rm Ta}_p^{}(J_s)^\dagger)
\longrightarrow H^1(\mathfrak{G}_{H_{cp^{t+1}}},\mathbf{T}^\dagger).
\]
By inflation, we shall view $\mathfrak{X}_{cp^{t+1}}$ as a class in $H^1(H_{cp^{t+1}},\mathbf{T}^\dagger)$.
\end{defn}

By \cite[Prop.~2.3.1]{howard-invmath}, the classes
\begin{equation}\label{eq:bigHP}
\mathfrak{Z}_{c,t}:=U_p^{-t}\cdot\mathfrak{X}_{cp^{t+1}}\in H^1(H_{cp^t},\mathbf{T}^\dagger)
\end{equation}
are compatible under the corestriction maps, thus defining a class
\[
\mathfrak{Z}_{c,\infty}:=\varprojlim_t\mathfrak{Z}_{c,t}\in H^1_{\rm Iw}(H_{cp^\infty},\mathbf{T}^\dagger). %:=\varprojlim_t H^1(K_t,\mathbf{T}^\dagger)$.
\]

The maximal free quotient of ${\rm Gal}(H_{cp^\infty}/K)$ is the Galois group
$\Gamma$ of the anticyclotomic $\bZ_p$-extension $K_\infty/K$, and so
for every character $\chi$ of $\Delta_c:={\rm ker}({\rm Gal}(H_{cp^\infty}/K)\twoheadrightarrow\Gamma)$,
we obtain a class
\[
\mathfrak{Z}_{c,\infty}^\chi\in H^1_{\rm Iw}(K_\infty,\mathbf{T}^\dagger\otimes\chi).
\]
For $c=1$ and $\chi=\mathds{1}$ this is the Iwasawa cohomology class $\mathfrak{Z}_\infty$
considered in \cite[\S{3.3}]{howard-PhD-I}. %in which case we shall use the same notation.

\section{Explicit reciprocity law}
\label{sec:comparison}

As in previous sections, let $\Gamma={\rm Gal}(K_\infty/K)$ (resp. $G={\rm Gal}(L_\infty/\bQ_p)$)
be the Galois group of the anticyclotomic $\bZ_p$-extension of $K$ (resp. the unique $\bZ_p^2$-extension of $\bQ_p$).
We assume that $p=\pp\overline{\pp}$ splits in $K$ %and does not divide the class number of $K$,
%so that $K_\infty/K$ is totally ramified at the primes above $p$,
and for each $v\mid p$ in $K$ with a slight abuse of notation
we let $K_{\infty,v}$ be the completion of $K_\infty$ at a fixed prime above $v$.

\subsection{Regulator map for the anticyclotomic $\bZ_p$-extension of $K$}

Recall the $\cR$-adic Hecke character $\boldsymbol{\xi}:K^\times\backslash\mathbb{A}_K^\times\rightarrow\cR^\times$
in (\ref{def:xi}), based on the choice of an $\ro$-valued Hecke character
$\lambda$ of $K$ of conductor prime to $Np$ and infinity type $(1,0)$.
With a slight abuse of notation, we also let $\boldsymbol{\xi}:G_K\rightarrow\cR^\times$ be the character
defined by
\[
\boldsymbol{\xi}(\sigma):=[\langle\widehat{\lambda}(\sigma)\widehat{\lambda}^{-1}(\tau\sigma\tau)\rangle^{1/2}]
\]
where $\tau\in G_K$ is the nontrivial automorphism of ${\rm Gal}(K/\bQ)$,
and set
\begin{equation}\label{def:tw}
\mathbb{T}:=\mathbf{T}\vert_{G_K}\otimes\Theta^{-1}\boldsymbol{\xi}^{-1}\varepsilon_{\rm cyc}^{-1}.
%\quad\quad\T:=\mathbf{T}\otimes\Theta^{-1}\xi. %=\mathbf{T}^\dagger\otimes\xi.
\end{equation}

By Theorem~\ref{thm:MT}, if $v$ is a place of $K$ above $p$,
the restriction of $\bT$ to a decomposition group at $v$
takes the form
\begin{equation}\label{eq:ord}
\bT_{\vert{G_{K_v}=G_{\bQ_p}}}\;:\;
\begin{pmatrix}
\alpha^{-1}\Theta\boldsymbol{\xi}^{-1} & * \\
0 & \alpha\Theta^{-1}\boldsymbol{\xi}^{-1}\varepsilon_{\rm cyc}^{-1}
\end{pmatrix}
\end{equation}
on a suitable $\cR$-basis. Since $\alpha^{-1}\Theta\boldsymbol{\xi}^{-1}$ is an unramified character of $G_{K_{\pp}}$,
the representation $(\ref{eq:ord})$ for $v=\pp$ is an ordinary deformation in the sense of
Definition~\ref{def:ochiai}, and hence associated with it
we may consider the regulator map $\mathcal{L}^G$ %:H^1_{\rm Iw}(L_\infty,\fil^+(\bT))\rightarrow\tilde{\cR}[[G]]$
of Theorem~\ref{thm:Log}.
\sk

In the following, we let $\bT$ be the representation $(\ref{eq:ord})$ for $v=\pp$ and
identify ${\rm Gal}(K_{\infty,\pp}/K_{\pp})$ with $\Gamma$ via $G_{\bQ_p}=G_{K_{\pp}}\hookrightarrow G_K$.
Recall the module $\mathbb{D}$ of Definition~\ref{def:D}.

%\begin{lem}
%The module $\mathcal{D}\widehat{\otimes}_{\bZ_p}\cS_\infty$ is free of rank $1$ over $\cR[[G]]$.
%\end{lem}

%\begin{proof}
%By \cite[Lemma~3.3]{Ochiai-Col}, the module
%$\mathcal{D}$ %=(\fil^+(\bT)\hat{\otimes}_{\bZ_p}\hat{\bZ}_p^{\rm nr})^{G_{\bQ_p}}\hat{\otimes}_{\bZ_p}\bZ_p[[\Gamma_{\rm cyc}]]$
%is free of rank $1$ over $\bZ_p[[\Gamma_{\rm cyc}]]$. Combined with Proposition~\ref{lem:Yager}, the result follows.
%\end{proof}

\begin{lem}\label{lem:ohta-wake}
There exists an element $\omega_{\F}^\vee\in\mathbb{D}$
%(\fil^+\mathbf{T}\otimes\Theta^{-2}\varepsilon_{\rm cyc}^{-1}\widehat{\otimes}_{\bZ_p}\widehat{\bZ}_p^{\rm nr})^{G_{\bQ_p}}$
such that
\[
\langle\k(\omega_\F^\vee),\omega_{f_\k}\rangle_{}=1
\]
for all $\k\in\mathcal{X}_{\ro}^a(\cR)$, where $\omega_{f_\k}\in\mathbf{D}_{\rm dR}(\fil^+V_\k)$ is
the differential associated to the $p$-stabilized newform $f_\k$.
\end{lem}

\begin{proof}
This is \cite[Prop.~10.1.1(1)]{KLZ2}.
\end{proof}

\begin{prop}\label{anticylo-reg}
Let $\lambda:=\mathbf{a}_p\cdot\Theta\boldsymbol{\xi}^{-1}({\rm Fr}_p)-1\in\cR$
and set $\widetilde{\cR}:=\cR[\lambda^{-1}]\otimes_{\bZ_p}\widehat{\cO}_{F_\infty}$.
There exists an $\widetilde{\cR}[[\Gamma]]$-linear map
\[
{\rm tw}_{-1}\mathcal{L}_{\omega_\F}^{\Gamma}:H^1_{\rm Iw}(K_{\infty,{\pp}},\fil^+\bT(1))
\longrightarrow\widetilde{\cR}[[\Gamma]]
\]
with the following interpolation property. Let $\mathfrak{Y}_\infty\in H^1_{\rm Iw}(K_{\infty,{\pp}},\fil^+\bT(1))$
and let $\widehat\phi:\Gamma\rightarrow L^\times$ be the $p$-adic avatar of
an anticyclotomic Hecke character of $K$ of conductor $p^n$ and infinity type $(\ell,-\ell)$.
\begin{enumerate}
\item[$(i)$]{} If $\k\in\mathcal{X}_{\ro}^a(\cR)$ is non-exceptional and $\ell\leqslant 0$, then:
\[
{\rm tw}_{-1}\mathcal{L}_{\omega_\F}^\Gamma(\mathfrak{Y}_\infty)(\k,\widehat\phi)=
\frac{\varepsilon(\widehat\phi^{-1})}{\k(\mathbf{a}^n_p)\chi_\k\x_\k^{-1}(p_{\pp}^n)}
\cdot\frac{\mathscr{P}^*(\k,\phi_{}^{-1})}{\mathscr{P}(\k,\phi_{})}
\cdot\frac{(-1)^{\ell}}{(-\ell)!}\cdot{\rm log}_{/\omega_{f_\k}}(\k(\mathfrak{Y}_\infty)^{\widehat\phi_{}^{-1}});
\]
\item[$(ii)$] If $\k\in\mathcal{X}_{\ro}^a(\cR)$ is $p$-old and $\ell>0$, then:
\[
{\rm tw}_{-1}\mathcal{L}_{\omega_\F}^\Gamma(\mathfrak{Y}_\infty)(\k,\widehat\phi_{})=
\frac{\varepsilon(\widehat\phi^{-1}_{})}{\k(\mathbf{a}^n_p)\chi_\k\x_\k^{-1}(p_\pp^n)}
\cdot\frac{\mathscr{P}^*(\k,\phi^{-1})}{\mathscr{P}(\k,\phi_{})}
\cdot(\ell-1)!\cdot{\rm exp}^*_{/\omega_{f_\k}}(\k(\mathfrak{Y}_\infty)^{\widehat\phi_{}^{-1}}),
\]
\end{enumerate}
where
\begin{align*}
\frac{\mathscr{P}^*(\k,\phi^{-1})}{\mathscr{P}(\k,\phi)}&=
\left\{
\begin{array}{ll}
\frac{\bigl(1-\k(\mathbf{a}_p)\chi_\k\xi_\k^{-1}\phi^{-1}(\pp)\bigr)}
{\bigl(1-p\k(\mathbf{a}_p)^{-1}\chi_\k\xi_\k^{-1}\phi(\pp)\bigr)}
&\textrm{if $n=0$;}\\
1&\textrm{if $n\geqslant 1$,}
\end{array}
\right.
%\mathscr{P}(\k,\phi)&=
%\left\{
%\begin{array}{ll}
%\bigl(1-p\k(\mathbf{a}_p)^{-1}\chi_\k\xi_\k^{-1}\phi(\pp)\bigr) & \textrm{if $n=0$;}\\
%1& \textrm{if $n\geqslant 1$},
%\end{array}
%\right.
\end{align*}
and ${\rm log}_{/\omega_{f_\k}}$ and ${\rm exp}^*_{/\omega_{f_\k}}$
denote the Bloch--Kato logarithm and dual exponential maps paired against $\omega_{f_\k}$.
\end{prop}

\begin{proof}
In view of $(\ref{eq:ord})$, the action of $G_{\bQ_p}$ on $\fil^+\bT$ is given by the unramified character
sending %a geometric Frobenius
${\rm Fr}_p$ to $\mathbf{a}_p\cdot\Theta\boldsymbol{\xi}^{-1}({\rm Fr}_p)
=\mathbf{a}_p\cdot\boldsymbol{\chi}\boldsymbol{\xi}^{-1}(p_\pp)$, and so by
Theorem~\ref{thm:Log} and Lemma~\ref{lem:ohta-wake} we may consider the map
\[
\mathcal{L}_{\omega_\F}^G:H^1_{\rm Iw}(L_\infty,\fil^+\bT)\longrightarrow\widetilde{\cR}[[G]]
\]
defined by the rule
\[
\mathcal{L}^G(-)=\mathcal{L}_{\omega_\F}^G(-)\cdot({\omega}_\F^\vee\otimes 1).
\]

Let ${\rm tw}_{-1}\mathcal{L}^G_{\omega_{\mathbf{f}}}$ be the composite map
\begin{align*}
{\rm tw}_{-1}\mathcal{L}^G_{\omega_{\mathbf{f}}}
:H^1_{\rm Iw}(L_\infty,\fil^+\bT(1))
%\simeq H^1_{\rm Iw}(L_\infty,\fil^+{\mathbb{T}})\otimes\bZ_p(1)
&\xrightarrow{\otimes (\zeta_{p^r})_r^{-1}}
H^1_{\rm Iw}(L_\infty,\fil^+\bT)
\xrightarrow{\mathcal{L}_{\omega_{\F}}^G}\widetilde{\cR}[[G]]
\xrightarrow{{\rm Tw}_{\varepsilon_{\rm cyc}^{-1}}}\widetilde{\cR}[[G]],
\end{align*}
where
\begin{itemize}
\item{} $(\zeta_{p^r})$ %:H^1_{\rm Iw}(L_\infty,\fil^+({\bT}))\rightarrow H^1_{\rm Iw}(L_\infty,\fil^+({\bT}))\otimes\bZ_p(1)$
is the $\cR$-linear map induced by $(x_r\;{\rm mod}\;p^r)\mapsto (x_r\otimes\zeta_{p^r}\;{\rm mod}\;p^r)$;
%the middle isomorphism is the composition of the isomorphism
%$H^1_{\rm Iw}(L_\infty,\fil^+({\bT}))\otimes\bZ_p(1)\cong H^1_{\rm Iw}(L_\infty,\fil^+({\bT})\otimes\bZ_p(1))$ (see \cite[Prop.~6.2.1(i)]{Rubin-ES})
%with the identification ${\bT}\otimes\bZ_p(1)\cong\T$,
\item{} ${\rm Tw}_{\varepsilon_{\rm cyc}^{-1}}$ is the $\cR$-linear
isomorphism given by $g\mapsto\varepsilon_{\rm cyc}^{-1}(g)g$ for $g\in G$,
\end{itemize}
and let $\mathbb{J}$ be the kernel of the natural projection
$\cR[[G]]\twoheadrightarrow\cR[[\Gamma]]$. The corestriction map
\[
H^1_{\rm Iw}(L_\infty,\fil^+\bT(1))/\mathbb{J}\longrightarrow
H^1_{\rm Iw}(K_{\infty,\pp},\fil^+\bT(1))
\]
is injective, and its cokernel is $H^2(L_\infty,\fil^+\bT(1))[\mathbb{J}]$, which
vanishes since $H^0(K_{\infty,\pp},\fil^+\bT(1))=\{0\}$ (as one can see e.g. by the argument
right before \cite[Lem.~5.5]{cas-hsieh1}). Quotienting ${\rm tw}_{-1}\mathcal{L}_{\omega_\F}^G$
by $\mathbb{J}$ we thus obtain a map
\[
{\rm tw}_{-1}\mathcal{L}_{\omega_\F}^\Gamma:
H^1(K_{\infty,\pp},\fil^+\bT(1))\simeq H^1(L_\infty,\fil^+\bT(1))/\mathbb{J}\longrightarrow\widetilde{\cR}[[\Gamma]]
\]
with the desired properties following from Theorem~\ref{thm:Log} and Corollary~\ref{thm:ERL}.
\end{proof}

\subsection{Explicit reciprocity law for big Heegner points}

Let $\mathscr{L}_{\pp,\xi}(\mathbf{f})$ be the two-variable $p$-adic $L$-function
constructed in $\S\ref{subsec:Lp}$, and let $c\cO_K$ be the conductor of $\lambda^{1-\tau}$,
where $\lambda$ is as used in the construction of the $\cR$-adic character $\bx$.
On the other hand, let $\mathfrak{Z}_{c,\infty}\in H^1_{\rm Iw}(H_{cp^\infty},\mathbf{T}^\dagger)$
be Howard's systems of big Heegner points as recalled in $\S\ref{subsec:bigHP}$.
As shown in the proof of \cite[Prop.~2.4.5]{howard-invmath}, for each place $v\mid p$
the restriction ${\rm res}_v(\mathfrak{Z}_{c,\infty})$ lies in the kernel of the natural map
\[
H^1_{\rm Iw}(H_{cp^\infty,v},\mathbf{T}^\dagger)
\longrightarrow H^1_{\rm Iw}(H_{cp^\infty,v},\fil^-\mathbf{T}^\dagger)
\]
induced by $(\ref{eq:ordGr2})$. In particular, by \cite[Lem.~2.4.4]{howard-invmath} the class
${\rm res}_\pp(\mathfrak{Z}_\infty)$ is the image of
a unique class in $H^1_{\rm Iw}(H_{cp^\infty,\pp},\fil^+\mathbf{T}^\dagger)$ which we shall
still denote in the same manner. Moreover, since
$\mathbf{T}^\dagger\otimes\boldsymbol{\xi}^{-1}=\bT(1)$ by (\ref{def:tw}),
the twist $\mathfrak{Z}_{c,\infty}^{\boldsymbol{\xi}^{-1}}$ lies in $H^1_{\rm Iw}(H_{cp^\infty},\bT(1))$;
%since clearly $H^0(K_{\infty,v},\fil^-\bT(1))=0$, we have
in the following, we let $\mathfrak{Z}_{c,\infty}^{\boldsymbol{\xi}^{-1}}$ be the image of this
class in $H^1_{\rm Iw}(K_{\infty},\bT(1))$ under corestriction, so that in particular
we have ${\rm res}_\pp(\mathfrak{Z}_{c,\infty}^{\boldsymbol{\xi}^{-1}})\in H^1_{\rm Iw}(K_{\infty,\pp},\fil^+\bT(1))$.

\begin{thm}\label{thm:equality}
%Let $\mathfrak{Z}_\infty$ be Howard's system of big Heegner points $(\ref{eq:bigHP})$,
%and let $\mathscr{L}_{\pp,\xi}(\mathbf{f})$ be the two-variable $p$-adic $L$-function
%of Definition~\ref{def:2varL}. Then
The following equality holds in $\widetilde{\cR}[[\Gamma]]$:
\[
{\rm tw}_{-1}\mathcal{L}_{\omega_{\F}}^\Gamma({\rm res}_\pp(\mathfrak{Z}_{c,\infty}^{\boldsymbol{\xi}^{-1}}))
=\mathscr{L}_{{\pp},\bx}(\F)\cdot\sigma_{-1,\pp},
\]
where $\sigma_{-1,\pp}:={\rm rec}_\pp(-1)\vert_{K_\infty}\in\Gamma$.
\end{thm}

%By Proposition~\ref{anticylo-reg},
The proof of Theorem~\ref{thm:equality} will be an immediate consequence of the following result.

\begin{prop}\label{prop:t}
Let $\k\in\mathcal{X}_{\ro}^a(\cR)$ be an arithmetic prime of weight $(2,\varepsilon)$
with $\varepsilon:\Gamma^{\rm wt}\rightarrow\boldsymbol{\mu}_{p^\infty}$ of conductor $p^s$,
and let $\widehat{\phi}:\Gamma\rightarrow L^\times$
be the $p$-adic avatar of an anticyclotomic Hecke character $\phi$ of $K$ of infinity type $(1,-1)$ and
conductor $p^n$. If $n\geqslant s$, then
\[
\mathscr{L}_{\pp,\bx}(\F)(\k,\widehat{\phi}^{-1}_{})=
%\frac{\k(\mathbf{a}^t_p)\cdot\chi_\k\x_\k^{-1}(p_{\pp}^n)}{\varepsilon(\widehat{\phi}_{\pp})}
%\cdot{\rm log}_{/\omega_{f_\k}}({\rm res}_{{\pp}}(\k(\mathfrak{Z}_{\infty})^{\xi_\k^{-1}\widehat{\phi}})).
\frac{\phi_\pp(-1)\varepsilon(\phi_\pp)}{\k(\mathbf{a}_p^n)\chi_\k\xi_\k^{-1}({\rm Fr}_p^n)}
\cdot{\rm log}_{/\omega_{f_\k}}({\rm res}_{\pp}(\k(\mathfrak{Z}_{c,\infty})^{\x_\k^{-1}\widehat{\phi}})).
\]
\end{prop}

\begin{proof}
Our hypotheses imply that the character $\xi_\k\phi^{-1}$ has finite order and it
factors through the ${\rm Gal}(H_{cp^{n+1}}/K)$.
By the same calculation as in the proof of \cite[Thm.~4.9]{cas-hsieh1}
(see esp. [\emph{loc.cit.}, (4.8)]) we obtain
\begin{equation}\label{eq:calc1}
\mathscr{L}_{\pp,\bx}(\F)(\k,\widehat\phi^{-1})
=\mathfrak{g}(\phi_{\pp}^{-1})p^{-n}\phi_\pp(p^n)
\sum_{\sigma\in{\rm Gal}(H_{cp^{n+1}}/K)}\xi_\k^{-1}\phi(\sigma)\chi_\k^{-1}(\sigma)\cdot
d^{-1}\widehat{f}_\k^{\flat}(P_{cp^{n+1},s}^\sigma),
\end{equation}
where $d^{-1}\widehat{f}_\k$ is the $p$-adic modular form
\[
d^{-1}\widehat{f}_\k^\flat:=\lim_{t\to -1}d^{t}\widehat{f}_\k^\flat=\sum_{(n,p)=1}\k(\mathbf{a}_n)n^{-1}q^n.
\]

To proceed with the proof, we need to recall the definition of the Frobenius operator
${\rm Frob}$ on the space $V_p(N;R)$ of $p$-adic modular forms, %(see \ref{sec:modular}),
where we take $R$ to be a complete discrete valuation ring containing $\cO_\k$.
If $x=[(A,\eta^{(p)},\eta_p)]$ is a point in $\widehat{\rm Ig}(N)_{/R}$ with
\[
(\eta^{(p)},\eta_p):\boldsymbol{\mu}_N\oplus\boldsymbol{\mu}_{p^\infty}\hookrightarrow A[N]\oplus A[p^\infty],
\]
then $\eta_p$ amounts to giving an isomorphism
$\widehat{\eta}_p:\widehat{A}\simeq\widehat{\mathbf{G}}_m$ of formal groups over $R$,
and we set
\[
{\rm Frob}(x):=(A_0,\eta_0^{(p)},\eta_{0,p}),
\]
where:
\begin{itemize}
\item{} $A_0:=A/\eta_p(\boldsymbol{\mu}_p)$ is the quotient of $A$ by its canonical subgroup;
\item{} $\eta_0^{(p)}:=\lambda_0\circ\eta^{(p)}:\boldsymbol{\mu}_N\hookrightarrow A_0[N]$,
where $\lambda_0:A\rightarrow A_0$ is the natural projection;
\item{} $\eta_{0,p}:\boldsymbol{\mu}_{p^\infty}\hookrightarrow A_0[p^\infty]$
induces $\widehat{\eta}_{0,p}:=\widehat{\eta}_p\circ\widehat{\mu}_0$, %:\hat{A}_0\simeq\hat{\mathbb{G}}_m$,
where $\widehat{\mu}_0:\widehat{A}_0\simeq\widehat{A}$ is the isomorphism of formal groups
induced by the dual isogeny $\mu_0=\lambda_0^\vee$.
\end{itemize}
The action of ${\rm Frob}$ on $V_p(N;R)$ is then defined in the obvious manner,
setting
\[
{\rm Frob}(g)(x):=g({\rm Frob}(x)),
\]
for every $g\in V_p(N;R)$ and $x\in\widehat{\rm Ig}(N)_{/R}$.

Now let $F_{\omega_{f_\k}}$ be the Coleman primitive of the differential
$\omega_{f_\k}$ vanishing at the cusp $\infty$;
this is a locally analytic $p$-adic modular form (as defined in \cite[p.~1083]{bdp1})
of weight $0$ satisfying
\[
d F_{\omega_{f_\k}}=\omega_{f_\k}
\]
and characterized by the further requirement that
\begin{equation}\label{eq:Colequation}
F_{\omega_{f_\k}}-\frac{\k(\mathbf{a}_p)}{p}{\rm Frob}(F_{\omega_{f_\k}})
=d^{-1}\widehat{f}_\k^\flat
\end{equation}
%where ${\rm d}^{-1}\hat{f}^\flat(q)=\sum_{p\nmid n}\k(\mathbf{a}_n)n^{-1}q^n$
(\emph{cf.} \cite[Cor.~2.8]{cas-inv}). In particular, note that
$U_pF_{\omega_{f_\k}}=\frac{\k(\mathbf{a}_p)}{p}F_{\omega_{f_\k}}$.

Let $F_{n,s}$ be a finite extension of $\imath_p(L_{cp^{n+1},s})$ in $\overline{\bQ}_p$
such that the base-change $X_s\times_{\bQ_p}F_{n,s}$ admits a stable model.
The calculation in \cite[Prop.~2.9]{cas-inv}
%of the $p$-adic Abel--Jacobi images of certain degree $0$ divisors on $X_s$, which
applies to $f$ and the classes
\[
\Delta_{cp^{n+1},s}:=(P_{cp^{n+1},s})-(\infty),\quad
\Delta_{cp^{n+1+s},s}:=(P_{cp^{n+1+s},s})-(\infty)
\]
in $J_s(F_{n,s})$, yielding the formulae
\begin{equation}\label{eq:higherAJ}
\begin{split}
{\rm log}_{\omega_{f_\k}}(\Delta_{cp^{n+1},s})
&=F_{\omega_{f_\k}}(P_{cp^{n+1},s}),\quad
{\rm log}_{\omega_{f_\k}}(\Delta_{cp^{n+1+s},s})
=F_{\omega_{f_\k}}(P_{cp^{n+1+s},s}),
\end{split}
\end{equation}
where ${\rm log}_{\omega_{f_\k}}:J_s(F_{n,s})\rightarrow\bC_p$ is
the formal group logarithm associated with $\omega_{f_\k}$.

Now define
$Q_{cp^{n+1},s}%^{\chi_\k}
\in J_s(L_{cp^{n+1},s})\otimes_{\bZ}F_\k$ by
\begin{equation}\label{eq:defQt}
Q_{cp^{n+1},s}=\sum_{\sigma\in{\rm Gal}(H_{cp^{n+1+s}}/H_{cp^{n+1}})}
\Delta_{cp^{n+1+s},s}^{\tilde{\sigma}}\otimes\chi_\k^{-1}(\tilde{\sigma}),
\end{equation}
where for each $\sigma\in{\rm Gal}(H_{cp^{n+1+s}}/H_{cp^{n+1}})$,
$\tilde{\sigma}$ is an arbitrary lift of $\sigma$ to ${\rm Gal}(L_{cp^{n+1},s}/H_{cp^{n+1}})$;
by $(\ref{eq:gal})$, the point $Q_{cp^{n+1},s}$ does not depend on the particular choice of lift.
Taking lifts $\tilde{\sigma}$ in $(\ref{eq:defQt})$ which act trivially on $\boldsymbol{\mu}_{p^s}$ (as we may, since
$H_{cp^{n+1+s}}\cap H_{cp^{n+1}}(\boldsymbol{\mu}_{p^s})=H_{cp^{n+1}}$) and extending the map
${\rm log}_{\omega_{f}}$ by $F_\k$-linearity, we deduce from $(\ref{eq:higherAJ})$ that
\begin{equation}\label{hor-comp}
\begin{split}
{\rm log}_{\omega_{f_\k}}({Q}_{cp^{n+1},s})
&=\sum_{\tau\in{\rm Gal}(L_{cp^{n+1},s}/H_{cp^{n+1}}(\boldsymbol{\mu}_{p^s}))}
F_{\omega_{f_\k}}(P_{cp^{n+1+s},s}^{\tau})\\
&=F_{{\omega_{f_\k}}}(U_p^s\cdot P_{cp^{n+1+s},s})\\
&=\left(\frac{\k(\mathbf{a}_p)}{p}\right)^s
\cdot F_{\omega_{f_\k}}(P_{cp^{n+1},s}),
\end{split}
\end{equation}
using Lemma~\ref{lem:points} for the second equality.
%Letting $\Delta^{\chi_\nu}_{p^{t+1},s}$ be the class ${}_{\bb}\Delta^{\chi_\nu}_{p^{t+1},s}$
%for $\bb=\O_{p^{t+1}}$,
Substituting $(\ref{hor-comp})$ into $(\ref{eq:calc1})$ and using $(\ref{eq:Colequation})$
we thus arrive at
\begin{equation}\label{eq:calc3}
\begin{split}
\mathscr{L}_{\pp,\bx}(\F)(\k,\widehat\phi^{-1})
&=\mathfrak{g}(\phi_{\pp}^{-1})p^{-n}\phi_\pp(p^n)\\
&\quad\times\left(\frac{p}{\k(\mathbf{a}_p)}\right)^s
\sum_{\sigma\in{\rm Gal}(H_{cp^{n+1}}/K)}\x_\k^{-1}\phi\chi_\k^{-1}(\sigma)
\cdot {\rm log}_{\omega_{f}}(Q_{cp^{n+1},s}^\sigma).
\end{split}
\end{equation}

Recall that $\mathbf{T}^\dagger$ denotes the twist $\mathbf{T}\otimes\Theta^{-1}$,
and note that $\mathbf{T}^\dagger\otimes_{\cR}F_\k\simeq\mathbf{T}\otimes_{\cR}F_\k$
as $G_{\bQ(\boldsymbol{\mu}_{p^s})}$-representations.
By Hida's control theorem (see e.g. \cite[Thm.~3.1(i)]{hida86b}),
the natural map $\mathbf{T}\rightarrow\mathbf{T}\otimes_{\cR}F_\k$ factors as
\[
\mathbf{T}\longrightarrow e^{\rm ord}{\rm Ta}_p(J_s)\longrightarrow\mathbf{T}\otimes_{\cR}F_\k,
\]
and tracing through the definition of $\mathfrak{X}_{cp^{n+1}}$ in $\ref{subsec:bigHP}$
we see that the image of ${Q}_{cp^{n+1},s}$ under the induced map
\begin{equation}\label{eq:inducedkum}
J_s(L_{cp^{n+1},s})\otimes_{}F_\k
%\xrightarrow{e^{\rm ord}}J_s^{\rm ord}(L_{p^{t+1},s})\otimes_{}F_\nu
\xrightarrow{e^{\rm ord}\circ{\rm Kum}}H^1(L_{cp^{n+1},s},e^{\rm ord}{\rm
Ta}_p(J_s)\otimes_{}F_\k)\longrightarrow H^1(L_{cp^{n+1},s},\mathbf{T}\otimes_{\cR}F_\k)\nonumber
\end{equation}
agrees with the image of $U_p^s\cdot\k(\mathfrak{X}_{cp^{n+1}})$
under the restriction
\[
H^1(H_{cp^{n+1+s}},\mathbf{T}^\dagger\otimes_{\cR}F_\k)\longrightarrow
H^1(L_{cp^{n+1},s},\mathbf{T}^\dagger\otimes_{\cR}F_\k)\simeq H^1(L_{cp^{n+1},s},\mathbf{T}\otimes_{\cR}F_\k),
\]
and hence
\begin{equation}\label{eq:diag}
{\rm log}_{\omega_{f}}(Q_{cp^{n+1},s})
%=\langle{\rm log}_{V_{\nu}^\dagger}({\rm Kum}_s(e^{\rm
%ord}{Q}_{p^{t+1},s}^{\chi_\nu})),\omega_{\F_\nu^*}\rangle_{\rm dR}
=\left(\frac{\k(\mathbf{a}_p)}{p}\right)^s\cdot{\rm log}_{/\omega_f}
({\rm res}_{{\pp}}(\k(\mathfrak{X}_{cp^{n+1}})))
\end{equation}
by the compatibility between the map $\log_{/\omega_f}$
in Proposition~\ref{anticylo-reg} and $\log_{\omega_f}$ (see \cite[\S{3.10.1}]{BK}).

Note that $\varepsilon(\phi_\pp)=\mathfrak{g}(\phi_\pp^{-1})\phi_\pp(-p^n)$. Thus
substituting $(\ref{eq:diag})$ into $(\ref{eq:calc3})$ and using $(\ref{eq:bigHP})$ for the second equality,
we conclude that
\begin{align*}
\mathscr{L}_{\pp,\boldsymbol{\xi}}(\F)(\k,\widehat{\phi}^{-1})
&=\frac{\phi_\pp(-1)\varepsilon(\phi_\pp)}{\chi_\k\xi_\k^{-1}(p_\pp^n)}
\sum_{\sigma\in{\rm Gal}(H_{cp^{n+1}}/K)}\x_\k^{-1}\phi(\sigma)\cdot
{\rm log}_{/\omega_{f_\k}}({\rm res}_{{\pp}}(\k(\mathfrak{X}_{cp^{n+1}})^\sigma))\\
%&=\frac{\k(\mathbf{a}_p^t)\cdot\x_\k\chi_\k^{-1}(p_{{\pp}}^t)}{\varepsilon(\phi_{\pp})}
%\sum_{[\bb]\in{\rm Pic}(\O_{p^{t+1}})}\x_\k\chi_\k\phi(\bb)\cdot
%{\rm log}_{/\omega_{\F_\k}}({\rm res}_{\pp}(\k(\mathfrak{Z}_{t}^{\sigma_\bb^{-1}})))\\
&=\frac{\phi_\pp(-1)\varepsilon(\phi_\pp)}{\k(\mathbf{a}_p^n)\chi_\k\xi_\k^{-1}(p_\pp^n)}
\cdot{\rm log}_{/\omega_{f_\k}}({\rm res}_{\pp}(\k(\mathfrak{Z}_{c,\infty})^{\x_\k^{-1}\widehat{\phi}})),
\end{align*}
as was to be shown.
\end{proof}

\begin{proof}[Proof of Theorem~\ref{thm:equality}]
In light of Proposition~\ref{anticylo-reg}, the content of Proposition~\ref{prop:t}
amounts to the equality
\[
{\rm tw}_{-1}\mathcal{L}_{\omega_{\F}}^\Gamma({\rm res}_\pp(\mathfrak{Z}_{c,\infty}^{\boldsymbol{\xi}^{-1}}))(\k,\widehat{\phi}^{-1})
=(\mathscr{L}_{{\pp},\bx}(\F)\cdot\sigma_{-1,\pp})(\k,\widehat{\phi}^{-1}),
\]
for all pairs $(\k,\phi)$ as in the statement of that result.
Since an element in $\widetilde{\cR}[[\Gamma]]$ is uniquely determined by its values
at such set of pairs, the result follows.
\end{proof}

\section{Applications}
\label{sec:arith-applic}

\subsection{Preparations}

Let $f\in S_{k}(\Gamma_1(N))$ be a $p$-ordinary newform of weight $k>0$
and level $N$ prime to $p$ defined over a finite extension $L$ of $\bQ_p$,
and let $K$ be an imaginary quadratic field satisfying hypothesis (\ref{eq:heeg-hyp}) relative to $N$
and in which $p=\pp\overline{\pp}$. Let $\chi$ be the $p$-adic avatar of an anticyclotomic Hecke character of $K$
of infinity type $(j,-j)$ with $j-k/2\in\bZ$, and set
\[
V_{f,\chi}:=V_f(k/2)\vert_{G_K}\otimes\chi.
\]
For $S$ a finite set of places of $K$ containing the primes above $Np$, and for
every finite extension $F/K$ inside $\overline{\bQ}$,
let $\mathfrak{G}_{F,S}$ be the Galois group of the maximal extension of $F$ unramified outside
the places above $S$. Recall that the \emph{Bloch--Kato Selmer group} ${\rm Sel}(F,V_{f,\chi})$
%for the $G_F:={\rm Gal}(\overline{\bQ}/F)$-representation $V_{f,\chi}=V_{f}(r)\otimes\chi$
is defined by
\begin{equation}\label{def:Sel}
{\rm Sel}(F,V_{f,\chi})={\rm ker}\biggl(H^1(\mathfrak{G}_{F,S},V_{f,\chi})%\xrightarrow{\oplus_v{\rm res}_v}
\longrightarrow\prod_{v}
\frac{H^1(F_v,V_{f,\chi})}{H^1_{\rm f}(F_v,V_{f,\chi})}\biggr),
\end{equation}
where $v$ runs over all places of $F$, and
%$H_{\rm f}^1(F_v,V_{f,\chi})\subseteq H^1(F_v,V_{f,\chi})$
%is \emph{finite subspace} of Bloch--Kato \cite{BK} given by
\[
H^1_{\rm f}(F_v,V_{f,\chi}):=
\left\{
\begin{array}{ll}
{\rm ker}\left(H^1(F_v,V_{f,\chi})\longrightarrow H^1(F_v^{\rm ur},V_{f,\chi})\right)&\textrm{if $v\nmid p$};
\\
{\rm ker}\left(H^1(F_v,V_{f,\chi})\longrightarrow H^1(F_v,V_{f,\chi}\otimes_{\bQ_p}\mathbf{B}_{\rm cris})\right)&
\textrm{if $v\mid p$}.
\end{array}
\right.
\]

For $T_{f,\chi}:=T_f(k/2)\vert_{G_K}\otimes\chi$ with $T_{f}\subseteq V_f$ a $G_\bQ$-stable lattice,
we define ${\rm Sel}(F,T_{f,\chi})$ by the same recipe $(\ref{def:Sel})$,
replacing $H^1_{\rm f}(F_v,V_{f,\chi})$ by their natural preimages in $H^1(F_v,T_{f,\chi})$.
By abuse of notation, we let $F_{\pp}$ denote the completion of $F$ at any place above $\pp$,
and similarly for $F_{\overline\pp}$.

\begin{lem}\label{HT}
%If $f$ has weight $2r\geqslant 2$ and
If the infinity type of $\chi$ is $(j,-j)$ with $j-k/2\in\bZ$ and $j\geqslant k/2$, then:
\[
H^1_{\rm f}(F_{\overline\pp},V_{f,\chi})=\{0\},
\quad\quad
H^1_{\rm f}(F_{\pp},V_{f,\chi})=H^1(F_{\pp},V_{f,\chi}).
\]
In particular, the classes in the Bloch--Kato Selmer group
${\rm Sel}(F,V_{f,\chi})$ are trivial at all primes above $\overline{\pp}$
and satisfy no local condition at the primes above $\pp$.
\end{lem}

\begin{proof}
Following our conventions (see the footnote in Theorem~\ref{thm:Log}),
we find that the Hodge--Tate weights of $V_{\overline\pp}:=V_{f,\chi}\vert_{G_{F_{\overline{\pp}}}}$
are $k/2-j$ and $1-k/2-j$, which are non-positive integers under the above hypotheses,
it follows that ${\rm Fil}^0\mathbf{D}_{\rm dR}(V_{\overline\pp})=\mathbf{D}_{\rm dR}(V_{\overline\pp})$.
Similarly, the Hodge--Tate weights of $V_{\pp}:=V_{f,\chi}\vert_{G_{F_{\pp}}}$ are
the strictly positive integers $k/2+j$ and $1-k/2+j$, which implies that ${\rm Fil}^0\mathbf{D}_{\rm dR}(V_{\pp})=\{0\}$.
The result thus follows from \cite[Thm.~4.1(ii)]{BK}.
\end{proof}

We will also have use for the following generalized Selmer groups
obtained by changing the local condition at the places above $p$
in definition $(\ref{def:Sel})$. For $v\mid p$ and $\mathcal{L}_v\in\{\emptyset,{\rm Gr},0\}$,
set
\[
H^1_{\mathcal{L}_v}(F_v,V_{f,\chi}):=
\left\{
\begin{array}{ll}
H^1(F_v,V_{f,\chi})&\textrm{if $\mathcal{L}_v=\emptyset$;}\\
H^1(F_v,\fil^+V_{f,\chi})&\textrm{if $\mathcal{L}_v={\rm Gr}$;}\\
\{0\}&\textrm{if $\mathcal{L}_v=0$,}
\end{array}
\right.
\]
and for $\mathcal{L}=\{\mathcal{L}_v\}_{v\mid p}$, define
\begin{equation}\label{def:auxsel}
H^1_{\mathcal{L}}(F,V_{f,\chi}):={\rm ker}\biggl(H^1(\mathfrak{G}_{F,S},V_{f,\chi})
%\xrightarrow{\oplus_v{\rm res}_v}
\longrightarrow\prod_{v\nmid p}\frac{H^1(F_v,V_{f,\chi})}{H^1_{\rm f}(F_v,V_{f,\chi})}\times\prod_{v\mid p}
\frac{H^1(F_{v},V_{f,\chi})}{H_{\mathcal{L}_v}^1(F_{v},V_{f,\chi})}\biggr).\nonumber
\end{equation}

%Note that if $\mathcal{L}$ is given by
%$\mathcal{L}_v={\rm Gr}$ for all $v\vert p$, then this
%recovers the previous definition $(\ref{def:Sel})$ of ${\rm Sel}(K,V_{f,\chi})$,
%
In particular, by Lemma~\ref{HT} we have
\begin{equation}\label{incl}
{\rm Sel}(F,V_{f,\chi})=H^1_{\emptyset,0}(F,V_{f,\chi}).
\end{equation}

As in $\S\ref{subsec:Gal-Hida}$, let $\mathbf{f}\in\cR[[q]]$ be the $\cR$-adic newform of tame level $N$
attached to $f$, and let $\mathbf{T}$ be the associated big Galois representation.

\begin{lem}\label{lem:tam}
Let $F$ be a finite extension of $K$, and
let $v$ be a prime of $F$ above a prime $\ell$ dividing $(D_K,N)$.
If $\bar{\rho}_f$ is ramified at $\ell$,
then $H^1(F_v^{\rm ur},\mathbf{T}^\dagger)$ is $\cR$-torsion free.
\end{lem}

\begin{proof}
This is well-known; see e.g. \cite[Lem.~3.12]{buy-bigHP}.
\end{proof}

For $F/K$ a finite extension, let
${\rm Sel}_{\rm Gr}(F,\mathbf{T}^\dagger)\subseteq H^1(\mathfrak{G}_{F,S},\mathbf{T}^\dagger)$ be the
strict Greenberg Selmer group of \cite[Def.~2.4.2]{howard-invmath}.

\begin{prop}\label{thm:Selmer}
If $\bar\rho_f$ is ramified at every prime $\ell\mid (D_K,N)$,
then $\mathfrak{X}_c\in{\rm Sel}_{\rm Gr}(H_c,\mathbf{T}^\dagger)$ for all positive integers
$c$ prime to $N$.
\end{prop}

\begin{proof}
The proof of \cite[Prop.~2.4.5]{howard-invmath} shows that %for any $c$ coprime with $N$
the localization ${\rm loc}_v(\mathfrak{X}_c)$ of $\mathfrak{X}_c$ %(denoted $\mathfrak{X}_1$ in \emph{loc.cit.})
at any place $v$ of $H_c$ lies in the local subspace
$H_{\rm Gr}^1(H_{c,v},\mathbf{T}^\dagger)\subseteq H^1(H_{c,v},\mathbf{T}^\dagger)$
defining ${\rm Sel}_{\rm Gr}(H_c,\mathbf{T}^\dagger)$, except possibly at primes $v\mid\ell\mid N$
which are nonsplit in $K$, in which case it is shown that
\[
{\rm loc}_{v}(\mathfrak{X}_c)\in{\rm ker}\biggl(H^1(H_{c,v},\mathbf{T}^\dagger)
\longrightarrow\frac{H^1(H_{c,v}^{\rm ur},\mathbf{T}^\dagger)}{H^1(H_{c,v}^{\rm ur},\mathbf{T}^\dagger)_{\rm tors}}\biggr),
\]
where $H^1(H_{c,v}^{\rm ur},\mathbf{T}^\dagger)_{\rm tors}\subseteq H^1(H_{c,v}^{\rm ur},\mathbf{T}^\dagger)$
is the $\cR$-torsion submodule. %It $v$ lies over a prime $\ell$ inert in $K$, then $H_{K,v}=K_\ell$, and so
In light of Lemma~\ref{lem:tam}, the result follows.
\end{proof}

%\subsection{Proof of Theorem~A}
\subsection{Higher weight specializations of big Heegner points}

In this section we show the connection between the higher weight specializations
of big Heegner points and the \'etale Abel--Jacobi images of classical Heegner cycles \cite{nekovar302}.
%answering a question raised by Howard (see \cite[p.~93]{howard-invmath}).
A first result along these lines was obtained in
\cite{cas-inv} under a certain nonvanishing hypothesis (see [\emph{loc.cit.}, Thm.~5.11]).
In Theorem~\ref{thm:higher-wt-sp} below we remove that hypothesis, and find a relation between
the global cohomology classes themselves, rather than just their cyclotomic $p$-adic heights.
\sk

Let $f\in S_{2r}(\Gamma_0(N))$ be a $p$-ordinary newform of even weight $2r\geqslant 2$, and let $T_f\subseteq V_f$
be a Galois stable lattice in the %self-dual Tate twist of
the $p$-adic Galois representations $\rho_f$ attached to $f$.
Fix an integer $c$ prime to $p$, let $\Delta_c$ be the kernel of the projection
${\rm Gal}(H_{cp^\infty}/K)\twoheadrightarrow\Gamma$, and for every character $\chi$ of $\Delta_c$, set
\[
{\rm Sel}_{\rm Gr}(K_{\infty},T_{f}(r)\otimes\chi):=\varprojlim_t{\rm Sel}(K_{t},T_f(r)\otimes\chi),
\]
where $K_t$ is the subfield of $K_\infty$ of degree $p^t$ over $K$.

\begin{lem}\label{lem:loc-inj}
If $\bar{\rho}_{f}\vert_{G_K}$ is irreducible, then the restriction map
\[
{\rm res}_\pp:{\rm Sel}_{\rm Gr}(K_{\infty},T_{f}(r)\otimes\chi)
\longrightarrow H^1_{\rm Iw}(K_{\infty,\pp},\mathscr{F}^+T_{f}(r)\otimes\chi)
\]
is injective.
\end{lem}

\begin{proof}
Since ${\rm Sel}_{\rm Gr}(K_{\infty},T_{f}(r)\otimes\chi)\subseteq H_{\rm Iw}^1(K_\infty,T_{f}(r)\otimes\chi)$ is $\Lambda$-torsion-free
by our hypothesis (see \cite[Lem.~2.2.9]{howard-PhD-I}), %\cite[\S{1.3.3}]{PR:Lp}),
it suffices to show that ${\rm ker}({\rm res}_\pp)$ is $\Lambda$-torsion, or equivalently, that for infinitely many characters $\phi:\Gamma\rightarrow\boldsymbol{\mu}_{p^\infty}$ of $p$-power order,
the $\phi$-specialized map
\begin{equation}\label{eq:res-chi}
{\rm Sel}_{\rm Gr}(K,V_{f}(r)\otimes\chi\phi)\longrightarrow H^1(K_{\pp},\mathscr{F}^+V_f(r)\otimes\chi\phi)
\end{equation}
is injective. Let $\mathbf{f}\in\cR[[q]]$ be the $\cR$-adic newform attached to $f$, and let
$\k\in\mathcal{X}_{\ro}^a(\cR)$ be such that $\mathbf{f}_\nu$
is the ordinary $p$-stabilization of $f$. %and let $\mathbf{z}_f:=\k(\mathfrak{Z}_\infty)$.
By Corollary~\ref{cor:2varL}, we have $\k(\mathscr{L}_{\pp,\bx}(\mathbf{f}))(\phi)\neq 0$ for all
but finitely many $\phi$, and by Theorem~\ref{thm:equality} this shows that
${\rm res}_\pp(\nu(\mathfrak{Z}_{c,\infty})^{\chi\phi})\neq 0$ for all but finitely many $\phi$, %of $p$-power order,
where $\nu(\mathfrak{Z}_{c,\infty})^{\chi\phi}$ is the image of $\nu(\mathfrak{Z}_{c,\infty}^\chi)$
under the $\phi$-specialization map
\[
{\rm Sel}_{\rm Gr}(K_\infty,T_f(r)\otimes\chi)\longrightarrow{\rm Sel}_{\rm Gr}(K,V_f(r)\otimes\chi\phi).
\]
Since by the results of \cite[\S{2.3}]{howard-invmath} the class
$\nu(\mathfrak{Z}_{c,\infty})^{\chi\phi}$ is the class over $K$ of an Euler system for $T_f(r)\otimes\chi\phi$
in the sense of \cite[Def.~7.2]{cas-hsieh1} with the Bloch--Kato local condition,
by \cite[Thm.~7.7]{cas-hsieh1} we have the implication
\[
\nu(\mathfrak{Z}_{c,\infty})^{\chi\phi}\neq 0\quad\Longrightarrow\quad
{\rm Sel}_{\rm Gr}(K,V_{f}(r)\otimes\chi\phi)=L.\nu(\mathfrak{Z}_{c,\infty})^{\chi\phi}.
\]
We thus conclude that $(\ref{eq:res-chi})$ is injective, and so the result follows.
\end{proof}

We are now ready to prove a strong refinement of the main result of
\cite{cas-inv}, %(\emph{cf.} Theorem~A in the Introduction),
as advanced in \cite[\S{1}]{cas-hsieh1}.
Let $K$ be an imaginary quadratic field of odd discriminant $-D_K<0$
satisfying hypothesis (\ref{eq:heeg-hyp}) relative to $N$ and in which the prime
$p=\pp\overline{\pp}$ splits (with $p\nmid 6N$, as usual) and fix a positive integer $c$
prime to $Np$. Let $\alpha$ be the $p$-adic unit root of the
$p$-th Hecke polynomial of $f$, and denote by
\[
\mathbf{z}_{f,c,\alpha}\in H^1_{\rm Iw}(H_{cp^\infty},T_f(r))
\]
the $\Lambda$-adic class constructed from generalized Heegner cycles in \cite[\S{5.2}]{cas-hsieh1}
with tame conductor $c$, which defines a class in ${\rm Sel}_{\rm Gr}(H_{cp^\infty},T_f(r))$.
We refer the reader to \cite[p.~1250]{cas-inv} and
\cite[Eq.~(4.6)]{cas-hsieh1} %(taking $\chi=\mathds{1}$) %and $c=1$)
for the definition of the $p$-adic \'etale Abel--Jacobi images
\[
\Phi_{g,H_c}^{\rm \acute{e}t}(\Delta_{c,r}^{\rm heeg}),\;\Phi_{g,H_c}^{\rm \acute{e}t}(\Delta_{c,r}^{\rm bdp})\;\in
{\rm Sel}(H_c,T_g(r)).
\]
of classical %(as in \cite{nekovar302})
and generalized %(as in \cite{bdp1})
Heegner cycles, respectively, attached to an eigenform $g$ of weight $2r\geqslant 2$. (For $r=1$,
both reduce to Kummer images of classical Heegner points.) For the next theorem,
write $k_0\geqslant 2$ for the weight of $f$, let $\mathbf{f}=\sum_{n=1}^\infty\mathbf{a}_nq^n\in\cR[[q]]$
be the associated $\cR$-adic newform of tame level $N$, and let $\mathfrak{Z}_{c,\infty}$ be Howard's system
of big Heegner points attached to $\mathbf{f}$ and $K$, as recalled in $\S\ref{subsec:bigHP}$.

\begin{thm}\label{thm:higher-wt-sp}
Assume in addition that: %$\bar{\rho}_f$ is ramified at every prime $q\mid (D_K,N)$,
%$p$-distinguished, and its restriction to $G_K$ is irreducible.
\begin{itemize}
%%\item[({\rm triv})]
\item{}
$k_0\equiv 2\pmod{p-1}$;
\item{}
$\bar{\rho}_f$ is ramified at every prime $q\mid(D_K,N)$;
\item{}
$\bar{\rho}_f$ $p$-distinguished;
\item{}
$\bar{\rho}_f\vert_{G_K}$ is irreducible.
\end{itemize}
Then for all $\k\in\mathcal{X}_{\ro}^a(\cR)$ of weight $2r>2$ with $2r\equiv k_0\pmod{2(p-1)}$
and trivial nebentypus, we have
\begin{equation}\label{eq:iw-classes}
\k(\mathfrak{Z}_{c,\infty})\cdot c^{r-1}=\mathbf{z}_{f_\k,c,\alpha}\nonumber
\end{equation}
as elements in ${\rm Sel}_{\rm Gr}(H_{cp^\infty},T_{f_\k}(r))$, where $\alpha=\k(\mathbf{a}_p)$.
In particular, for all such $\k$ %\in\mathcal{X}_{\ro}^a(\cR)$
we have
\begin{equation}\label{eq:0-classes}
\k(\mathfrak{Z}_{c,0})=\biggl(1-\frac{p^{r-1}}{\k(\mathbf{a}_p)}\biggr)^2
\cdot\frac{\Phi^{\rm \acute{e}t}_{f_\k,H_c}(\Delta^{\rm heeg}_{r,c})}{u_c(2c\sqrt{-D_K})^{r-1}},
\end{equation}
where $u_c=\vert\cO_c^\times\vert/2$.
\end{thm}

\begin{proof}
Letting $\chi$ be any fix character of $\Delta_c$, it suffices to show the equality
for the corresponding classes in ${\rm Sel}(K_\infty,V_f(r)\otimes\chi)$.
Take $\bx$ so that $\psi:=\xi_\nu$ restricts to the character $\chi$ on $\Delta_c$,
%Let $L_\infty$ be the completion of $K_{c,\infty}$ at a fixed prime above $\pp$, and
and let
\[
\mathcal{L}_{\k,\pp}^{\psi}:=\left\langle\mathcal{L}_{\pp,\psi}(-),\omega_{f_\k}\otimes t^{1-2r}\right\rangle:
H^1_{\rm Iw}(K_{\infty,\pp},\mathscr{F}^+V_{f_\k}(r)\otimes\chi^{-1})\longrightarrow\Lambda_{R}(\Gamma)
\]
be the map introduced in \cite[\S{5.3}]{cas-hsieh1} twisted by $\chi^{-1}$.
Then the map ${\rm tw}_{-1}\mathcal{L}_{\omega_{\F}}^\Gamma$
of Proposition~\ref{anticylo-reg} twisted by $\bx^{-1}$ specializes at $\k$ to the map
$\mathcal{L}_{\k,\pp}$ twisted by $\psi^{-1}$, and so by
Theorem~\ref{thm:equality} we have the relations
\begin{equation}\label{eq:1}
%\left\langle\mathcal{L}_{\pp,\xi_\k}(\k(\mathfrak{Z}_\infty)),\omega_{f_\k}\otimes t^{1-2r}\right\rangle
\mathcal{L}_{\k,\pp}^\psi(\k(\mathfrak{Z}_{c,\infty}))
=\k({\rm tw}_{-1}\mathcal{L}_{\omega_{\F}}^\Gamma({\rm res}_\pp(\mathfrak{Z}_{c,\infty}^{\bx^{-1}})))
=\k(\mathscr{L}_{\pp,\bx}(\mathbf{f})\cdot\sigma_{-1,\pp}).
\end{equation}

On the other hand, as shown in the proof of Theorem~\ref{thm:bigLp},
%the %two-variable $p$-adic $L$-function
$\mathscr{L}_{\pp,\bx}(\mathbf{f})$
specializes at $\k$ to the $p$-adic $L$-function $\mathscr{L}_{\pp,\psi}(f_\k)$
of \cite[\S{3.3}]{cas-hsieh1}, and so by the
explicit reciprocity law of [\emph{loc.cit.}, Thm.~5.7] we have the equalities
\begin{equation}\label{eq:2}
\k(\mathscr{L}_{\pp,\bx}(\mathbf{f})\cdot\sigma_{-1,\pp})=\mathscr{L}_{\pp,\psi}(f_\k)\cdot\sigma_{-1,\pp}
=%\left\langle\mathcal{L}_{\pp,\xi_\k}(\mathbf{z}_{f_\k,\alpha}),\omega_{f_\k}\otimes t^{1-2r}\right\rangle,
\mathcal{L}_{\k,\pp}^\psi(\mathbf{z}_{f_\k,c,\alpha})\cdot c^{1-r},
\end{equation}
where $\alpha=\k(\mathbf{a}_p)$, since this is the $U_p$-eigenvalue of the $p$-stabilized newform $f_\k$.

Comparing $(\ref{eq:1})$ and $(\ref{eq:2})$, the proof of the first statement in Theorem~\ref{thm:higher-wt-sp}
follows from Lemma~\ref{lem:loc-inj} and the injectivity of $\mathcal{L}^\psi_{\k,\pp}$.
(The injectivity of this map is not explicitly stated in \cite[\S{5.3}]{cas-hsieh1}, but
it follows from the construction in [\emph{loc.cit.}, Thm.~5.1] and \cite[Prop.~4.11]{LZ2}.)
In particular, by the construction of $\mathbf{z}_{f_\k,c,\alpha}$
in \cite[\S{5.2}]{cas-hsieh1} (see [\emph{loc.cit.}, Def.~5.2]), we obtain the relation
\[
\k(\mathfrak{Z}_{c,0})=\frac{1}{u_c}\biggl(1-\frac{p^{r-1}}{\k(\mathbf{a}_p)}\biggr)^2
\cdot\Phi^{\rm \acute{e}t}_{f_\k,H_c}(\Delta^{\rm bdp}_{c,r}),
\]
where $u_c=\vert\cO_c^\times\vert/2$, and by \cite[Prop.~4.1.2]{bdp4} (with $r_1=2r-2$, $r_2=0$, and so $u=r-1$)
relation $(\ref{eq:0-classes})$ follows.
\end{proof}

\subsection{Proof of Theorem~C}%{Equivariant BSD conjecture for CM elliptic curves}

%\begin{thm}\label{thm:BSD-CM}
%\end{thm}

%In light of equality $(\ref{incl})$, the content of Theorem~A amounts to the implication
%\[
%L(f/K,\chi,r)\neq 0\quad\Longrightarrow\quad H^1_{\emptyset,0}(K,V_{f,\chi})=\{0\}.
%\]

With the same notations as in %the statement of Theorem~C in
the Introduction, let $V_\varrho^\vee$ denote the contragredient of the
representation $V_\varrho$, and let $g\in S_1(\Gamma_1(N_\varrho))$ be a cusp form
whose associated Deligne--Serre representation
$V_g$ is isomorphic to $V_\varrho^\vee$ (the existence of $g$ is a consequence of the proof \cite{KW-serre-I}
of Serre's modularity conjecture). If $\mathfrak{P}$ is prime of $E$ above $p$
as in the statement of Theorem~C, we shall view $g$ and $V_g$ as the defined over
the finite extension of $\bQ_p$ given by the completion $L:=E_{\mathfrak{P}}$,
and let $T_g\subseteq V_g$ be any $G_\bQ$-stable $\cO_L$-lattice.

Let $g_p\in S_1(\Gamma_0(p)\cap\Gamma_1(N_\varrho))$ be a $p$-stabilization
of $g$. By \cite[Thm.~3]{wiles88}, there exists an $\cR$-adic newform $\mathbf{f}$ of tame conductor $N_\varrho$
with $\nu_1(\mathbf{f})=g_p$, and our hypotheses on $\varrho$ guarantee that the associated residual representation
$\rho_{\mathbf{f}}$ is irreducible and $p$-distinguished. On the other hand,
let $\lambda$ be the grossencharacter associated to $A$ by the theory
of complex multiplication, and let $\mathscr{L}_{\pp,\bx}(\mathbf{f})$ be the two-variable
$p$-adic $L$-function of $\S\ref{subsec:Lp}$ constructed with the corresponding anticyclotomic $\cR$-adic
character $\bx$. If $\mathfrak{Z}_{c,\infty}\in{\rm Sel}_{\rm Gr}(H_{cp^\infty},\mathbf{T}^\dagger)$ is Howard's
system of big Heegner points attached to $\mathbf{f}$ and $K$, where as usual $c\cO_K$
is the conductor of $\lambda^{1-\tau}$, setting $\chi:=\xi_{\nu_1}$
we see from Theorem~\ref{thm:bigLp} and Theorem~\ref{thm:equality} that
\begin{equation}\label{eq:barp}
\begin{split}
L(A/\bQ,\varrho,1)\neq 0\;&\Longrightarrow\;
L(g/K,\chi\mathbf{N}^{-1/2},0)\neq 0\\
&\Longrightarrow\;\nu_1(\mathscr{L}_{\pp,\bx}(\mathbf{f}))(\mathds{1})\neq 0\\
&\Longrightarrow\;{\rm res}_\pp(\nu_1(\mathfrak{Z}_{c,\infty})^{\chi^{-1}}))\neq 0,
\end{split}
\end{equation}
and so ${\rm res}_{\overline\pp}(\nu_1(\mathfrak{Z}_{c,\infty})^{\chi})\neq 0$
by the action of complex conjugation. %; in particular, $\mathbf{z}_f^{\chi}\neq 0$.

By the Euler system relations %for Howard's big Heegner points
established in \cite[\S{2.3}]{howard-invmath},
%the twisted class
$\nu_1(\mathfrak{Z}_{c,\infty})^{\chi^{-1}}$ is the class over $K$ of an
anticyclotomic Euler system for $T_g(1/2)\otimes\chi$ in the sense of \cite[Def.~7.2]{cas-hsieh1}
satisfying the local conditions defining $H^1_{{\rm Gr},{\rm Gr}}(K,V_{g,\chi})$,
where $V_{g,\chi}:=V_g(1/2)\vert_{G_K}\otimes\chi$. Thus as in the proof of \cite[Thm.~7.9]{cas-hsieh1}
the last nonvanishing in $(\ref{eq:barp})$ implies that
\[
H^1_{{\rm Gr},{\rm Gr}}(K,V_{g,\chi})=L\cdot(\nu_1(\mathfrak{Z}_{c,\infty})^{\chi^{-1}})^\tau
=L\cdot\nu_1(\mathfrak{Z}_{c,\infty})^\chi,
\]
and since ${\rm res}_{\overline\pp}(\nu_1(\mathfrak{Z}_{c,\infty})^{\chi})\neq 0$, this implies that
\begin{equation}\label{fouquet}
H^1_{{\rm Gr},0}(K,V_{g,\chi})=\{0\}.
\end{equation}

From Poitou--Tate duality we obtain the exact sequence
\begin{align*}
0\longrightarrow H^1_{0,\emptyset}(K,V_{g,\chi^{-1}})\longrightarrow H^1_{{\rm Gr},\emptyset}(K,V_{g,\chi^{-1}})
&\xrightarrow{{\rm res}_{\pp}} H^1(K_{\pp},\fil^+V_{g,\chi^{-1}})\nonumber\\
&\longrightarrow H^1_{\emptyset,0}(K,V_{g,\chi})^\vee\longrightarrow H^1_{{\rm Gr},0}(K,V_{g,\chi})^\vee,
\end{align*}
and since $H^1(K_{\pp},\fil^+V_{g,\chi^{-1}})$ is one-dimensional,
combining $(\ref{eq:barp})$ and $(\ref{fouquet})$ we conclude that
\begin{equation}\label{eq:vanishing}
H^1_{\emptyset,0}(K,V_{g,\chi})=\{0\},
\end{equation}
and so ${\rm Sel}(K,V_p(A)\otimes V_\varrho^\vee)$ vanishes by Lemma~\ref{thm:Selmer}.

Now let $H={\rm Gal}(F/\bQ)$, where $F$ is the splitting field of $\varrho$.
Since ${\rm Hom}_{G_\bQ}(V_\varrho,{\rm Sel}(F,V_p(A))_L)$ can be identified with the set of $H$-invariant
classes in ${\rm Sel}(F,V_p(A))\otimes V_\varrho^\vee={\rm Sel}(F,V_p(A)\otimes V_\varrho^\vee)$,
and the restriction map
\[
{\rm Sel}(\bQ,V_p(A)\otimes V_\varrho^\vee)\longrightarrow{\rm Sel}(F,V_p(A)\otimes V_\varrho^\vee)^H
\]
is as isomorphism, the proof of Theorem~C follows immediately from $(\ref{eq:vanishing})$.
\sk

\noindent\emph{Acknowledgements.}
It is a pleasure to thank Daniele Casazza,
Henri Darmon, Daniel Disegni,
Ben Howard, Ming-Lun Hsieh, Antonio Lei, Jan Nekov{\'a}{\v{r}}, Victor Rotger,
Ari Shnidman, Chris Skinner, Eric Urban, and Xin Wan for useful conversations
and correspondence related to this work.
%Some of the results in this paper were first outlined %on April 23, 2012,
%at the workshop on the $p$-adic Langlands programme held
%at the Fields Institute of Toronto in April~2012, and we also thank
%the Fields Institute and
%the organizers of the workshop %, and particularly Jan Nekov{\'a}{\v{r}},
%for their hospitality and support.

\bibliographystyle{amsalpha}
\bibliography{Heegner} %insert the name of your .bib document

\end{document}